[1994/12/01]

\documentclass[11pt]{amsart}
\usepackage{amsmath,amsthm,amsfonts,amscd,amssymb,eucal,latexsym,mathrsfs}
\usepackage{amsmath,amsthm,amsfonts,amscd,amssymb,comment,eucal,latexsym,mathrsfs}
\usepackage[all]{xy}
\newtheorem{thm}{Theorem}[section]
\newtheorem{cor}[thm]{Corollary}
\newtheorem{lem}[thm]{Lemma}
\newtheorem{prop}[thm]{Proposition}

\theoremstyle{definition}
\newtheorem{defn}[thm]{Definition}

\newtheorem{example}[thm]{Example}

\DeclareMathOperator{\Aut}{Aut}
\DeclareMathOperator{\OAut}{OAut}
\DeclareMathOperator{\Ad}{Ad}

\DeclareMathOperator{\Homeo}{Homeo}
\DeclareMathOperator{\Perm}{Perm}
\DeclareMathOperator{\Ball}{Ball}
\DeclareMathOperator{\id}{id}
\DeclareMathOperator{\Ext}{Ext}

\DeclareMathOperator{\diam}{diam}
\DeclareMathOperator{\im}{im}
\DeclareMathOperator{\Inn}{Inn}
\DeclareMathOperator{\supp}{supp}

\newcommand{\cstar}{\ensuremath{C^{*}}}

\title{MF Actions and K-theoretic Dynamics}
\author{Timothy Rainone}

\begin{document}

\maketitle

\begin{abstract}
We study the interplay of \cstar-dynamics and $K$-theory. Notions of chain recurrence for transformations groups $(X,\Gamma)$ and MF actions for non-commutative \cstar-dynamical systems $(A,\Gamma,\alpha)$ are translated into $K$-theoretical language, where purely algebraic conditions are shown to be necessary and sufficient for a reduced crossed product to admit norm microstates. We are particularly interested in actions of free groups on AF algebras, in which case we prove that a $K$-theoretic coboundary condition determines whether or not the reduced crossed product is a Matricial Field (MF) algebra. One upshot is the equivalence of stable finiteness and being MF for these  reduced crossed product algebras.
\end{abstract}

\section{Introduction}

The theory of operator algebras and dynamical systems are closely interwoven. Crossed product algebras arising from transformation groups, and in general from non-commutative \cstar-dynamical systems,  play an important role in the study of \cstar-algebras. One would like to uncover information about the algebra by studying the dynamics and, conversely, describe the nature of the dynamics by looking at the operator algebra's structure and invariants.

We are particularly interested in finite-dimensional approximation properties of \cstar-algebras. While nuclearity and exactness are measure-theoretic concepts, residual finite dimensionality, quasidiagonality and admitting norm microstates are properties more topological in nature as they concern matricial approximation of both the linear and multiplicative structure of the algebra. In this work we flesh out the appropriate dynamical conditions that give rise to such topological approximations in resulting reduced crossed products, and give $K$-theoretic expression to these conditions when the underlying algebras have sufficiently many projections. One purpose of this paper is to provide a $K$-theoretic interpretation of dynamical approximation properties such as residual finiteness and quasidiagonal actions as introduced by Kerr and Nowak in~\cite{KN}, and by doing so, extend results found in~\cite{B} and ~\cite{Pi}.

In the classical setting, Pimsner described a purely topological dynamical property for a $\mathbb{Z}$-system $(X,\mathbb{Z})$ that renders the resulting crossed product $C(X)\rtimes_{\lambda}\mathbb{Z}$ AF-embeddable. He showed in~\cite{Pi} that for a self-homeomorphism $T$ of a compact metrizable space $X$ the following are equivalent: (1) the crossed product embeds into an AF algebra, (2) the crossed product is quasidiagonal, (3) the crossed product is stably finite, (4) ``$T$ compresses no open sets'', that is, there does not exist a non-empty open set $U\subset X$ for which $T(\overline{U})\subsetneqq U$, which is equivalent to the action being \emph{chain recurrent}, that is, for every $x\in X$ and every $\varepsilon>0$, there are finitely many points $x=x_{1},\dots,x_{n}=x$ such that $d(T(x_{j}),x_{j+1})<\varepsilon$ for $1\leq j\leq n$.

It was N. Brown who saw condition (4) as being essentially $K$-theoretical, at least in the presence of many projections~\cite{B}. When $X$ is zero-dimensional we have $K_{0}(C(X))=C(X;\mathbb{Z})$, and the chain recurrence condition is expressed as $\hat T(f)<f$ for no non-zero $f\in C(X;\mathbb{Z})$, where $\hat T:C(X;\mathbb{Z})\rightarrow C(X;\mathbb{Z})$ is the induced order automorphism given by $\hat T(f)=f\circ T^{-1}$. Brown was then able to generalize Pimsner's result to the non-commutative setting as follows.

\begin{thm}[Brown] Let $A$ be an AF algebra and $\alpha\in\Aut(A)$ an automorphism. Then the following are equivalent:
\begin{enumerate}
\item $A\rtimes_{\alpha}\mathbb{Z}$ is AF-embeddable.
\item $A\rtimes_{\alpha}\mathbb{Z}$ is quasidiagonal.
\item $A\rtimes_{\alpha}\mathbb{Z}$ stably finite.
\item The induced map $\hat\alpha:K_{0}(A)\rightarrow K_{0}(A)$ ``compresses no element'', that is, there is no $x\in K_{0}(A)$ for which $\hat\alpha(x)<x$; equivalently, $H_{\alpha}\cap K_{0}(A)^{+}=\{0\}$, where $H_{\alpha}$ is the coboundary subgroup $\{x-\hat\alpha(x)\ :\ x\in K_{0}(A)\}$.
\end{enumerate}
\end{thm}

One of the main results of this paper, Theorem~\ref{Main2}, extends Brown's result to the case of a free group on $r$ generators acting on a unital AF algebra. In this case the coboundary subgroup is given by $H_{\sigma}=\im(\sigma)$ where $\sigma:\oplus_{j=1}^{r}K_{0}(A)\rightarrow K_{0}(A)$ is the coboundary morphism in the Pimsner-Voiculescu six-term exact sequence. In abbreviated form our Theorem says the following.

\begin{thm} Let $A$ be a unital AF algebra and $\alpha:\mathbb{F}_{r}\rightarrow\Aut(A)$ an action of the free group on $r$ generators. Then the following are equivalent.
\begin{enumerate}
\item $H_{\sigma}\cap K_{0}(A)^{+}=\{0\}.$
\item The reduced crossed product $A\rtimes_{\lambda,\alpha}\mathbb{F}_{r}$ is MF.
\item The reduced crossed product $A\rtimes_{\lambda,\alpha}\mathbb{F}_{r}$ is stably finite.
\end{enumerate}

\end{thm}

In order to extend the results of Pimsner and Brown to actions of discrete countable groups, one needs the right notion of chain recurrence for arbitrary transformation groups and a corresponding approximation property for \cstar-dynamical systems. D. Kerr and P. Nowak then introduced residually finite actions and quasidiagonal actions in~\cite{KN} where it was shown that these systems give rise to MF crossed products provided that the reduced group \cstar-algebra of the acting group is itself MF. This is a necessary condition as being MF passes to subalgebras and the reduced group \cstar-algebras sits canonically inside the reduced crossed product. Thus one cannot hope for quasidiagonality, or much less AF-embeddability, when the acting group is non-amenable, for Rosenberg's result (\cite{HadR}) asserts that a discrete group whose reduced \cstar-algebra is quasidiagonal must be amenable.

Matricial field (MF) algebras were introduced in by Blackadar and Kirchberg in~\cite{BK}. These are stably finite \cstar- algebras which arise from generalized inductive limits of finite-dimensional algebras, or, equivalently, which admit norm microstates~\cite{BO}. The MF property is the \cstar-analogue of admitting tracial microstates, i.e., embeddability into the ultrapower $R^\omega$ of the hyperfinite II$_{1}$ factor. Blackadar and Kirchberg remarked that there is no example of a stably finite separable \cstar-algebra which is known not to be MF, but that a good candidate is $\cstar_{\lambda}(\mathbb{F}_{r})$. Then came the deep result of U. Haagerup and S. Thorbj{\o}rnsen in~\cite{HT} that showed that $\cstar_{\lambda}(\mathbb{F}_{r})$ is in fact MF. It therefore seems natural to focus our attention on actions of free groups on AF algebras. By studying the induced $K$-theoretic dynamics of such systems, purely algebraic conditions emerge which are necessary and sufficient for $A\rtimes_{\lambda,\alpha}\mathbb{F}_{r}$ to be MF, one in the form of locally invariant states on $K_{0}(A)$ and the other in the spirit of a coboundary subgroup as in Brown's work (Theorem 0.2 of~\cite{B}). These are summed up in Theorem~\ref{Main2} below. These $K$-theoretic conditions enable us to prove that being MF and being stably finite are equivalent for this class of crossed products.

MF algebras are interesting in their own right but are also important in Voiculescu's seminal study of topological free entropy dimension for a family of self-adjoint elements $a_{1},\dots,a_{n}$ in a unital \cstar-algebra $A$~\cite{V1}. Indeed the latter is well-defined only when $\cstar(\{a_{1},\dots,a_{n}\})$ is MF. There is also a connection between MF algebras and the Brown-Douglas-Filmore \emph{Ext} semigroup introduced in~\cite{BDF}. This note will exhibit several examples of MF algebras whose \emph{Ext} semigroup is not a group.

We briefly outline the flow and content of the paper. Section 2 introduces notation, relevant constructions and necessary concepts that will appear thereafter. In section 3 we explore MF, QD, and RFD actions in the commutative and non-commutative settings along with their crossed product structure. Loosely speaking, an MF action on a algebra is one that can be approximately modeled matricially. When the underlying algebra is nuclear, we show how the approximating dynamics completely characterize RFD, QD and MF reduced crossed products (see Theorem~\ref{Summarypart1}). For example, it is shown that given any action $\mathbb{F}_{r}\curvearrowright A$ on a UHF algebra, the reduced crossed product $A\rtimes_{\lambda}\mathbb{F}_{r}$ is always MF and $\Ext(A\rtimes_{\lambda}\mathbb{F}_{r})$ is not a group. Section 4 deals primarily with $K$-theoretic dynamics as we study residually finite, RFD and MF actions in their $K$-theoretic expressions. These systems will admit, in a local sense, invariant states on the $K_{0}$ group of the underlying algebra. This enables us to prove the main result, Theorem~\ref{Main2}, which determines exactly when the reduced crossed product of free group on a unital AF-algebra is MF.

The author would like to express his gratitude to his adviser David Kerr for his support and many suggestions.

\section{First principles and notation}

Throughout this paper $\Gamma$ will always denote a countable discrete group and $(X,d)$ a compact metrizable space and all $\cstar$-algebras $A$ will be considered separable and with unit $1_{A}$. We write $\mathcal{P}(A)$ for the set of projections in $A$ and we set $\mathcal{P}_{\infty}(A)=\bigcup_{n\geq1}\mathcal{P}(M_{n}(A))$. We will frequently encounter the free group $\mathbb{F}_{r}$ on $r$ generators. Recall that a group action is simply a group homomorphism $h:\Gamma\rightarrow\Perm(E)$ from a group $\Gamma$ to the group of permutations on an arbitrary set $E$. At times, for economy, we write $\Gamma\curvearrowright E$ to denote the action and $h_{s}(x)=s.x$ for $s\in\Gamma$ and $x\in X$. When $E$ has additional structure, e.g. when $E=X$ is a topological space, $E=A$ a \cstar-algebra or $E=(G,G^{+},u)$ an ordered abelian group, one imposes extra conditions on the action so that it respects the prescribed category.  More precisely, by a \emph{continuous action} $\Gamma\curvearrowright X$, or equivalently a \emph{transformation group} $(X,\Gamma)$, we mean a group homomorphism $h:\Gamma\rightarrow\Homeo(X)$ where $\Homeo(X)$ denotes the group of homeomorphisms of $X$. In an operator algebraic framework one speaks of a \emph{\cstar-dynamical system} $(A,\Gamma,\alpha)$, where $A$ is a \cstar-algebra, $\Gamma$ a topological group and $\alpha :\Gamma \rightarrow\Aut(A)$ a continuous group homomorphism into $\Aut(A)$; the topological group of automorphisms of $A$ with the point-norm topology. Again we emphasize that since $\Gamma$ is discrete, we need not worry about the continuity of $\alpha$. In the case where $A$ is a commutative algebra, say $A=C(X)$ for some compact Hausdorff space $X$, \cstar-systems $(C(X),\Gamma,\alpha)$ are in one-to-one correspondence with transformation groups $(X,\Gamma)$ via the formula $\alpha_{s}(f)(x)=f(s^{-1}.x)$ where $s\in\Gamma, f\in C(X), x\in X$. Given a $C^{*}$-dynamical system $(A,\Gamma,\alpha)$, we write $A\rtimes_{\alpha}\Gamma$ to denote the full crossed product \cstar-algebra whereas $A\rtimes_{\lambda,\alpha}\Gamma$ will stand for the reduced algebra. Recall that if $\Gamma$ is amenable then we have $A\rtimes_{\alpha}\Gamma=A\rtimes_{\lambda,\alpha}\Gamma$. Furthermore, if $\Gamma$ is amenable and $A$ is nuclear then $A\rtimes_{\lambda,\alpha}\Gamma$ is nuclear as well. We refer the reader to~\cite{BO}, \cite{Wi} and~\cite{Ph} for details concerning the construction and properties of crossed products.

A \cstar-dynamical system induces a natural action at the $K$-theoretical level, and the order theoretical dynamics will reflect information about the nature of the action and will often describe the structure of the crossed product. Recall that a unital stably finite \cstar-algebra $A$ yields an ordered abelian group $(K_{0}(A),K_{0}(A)^{+})$ with order unit $[1_{A}]$. If $(G, G^{+},u)$ and $(H,H^{+},v)$ are ordered abelian groups each with their distinguished order units, a morphism in this category is a group homomorphism $\beta:G\rightarrow H$ which is positive and order unit preserving, i.e. $\beta(G^{+})\subset H^{+}$, and $\beta(u)=v$ respectively. We also write
\[\OAut(G):=\{\tau\in\mbox{Aut}(G): \tau(G^{+})=G^{+}, \tau(u)=u\}\]
for the set of ordered abelian group automorphisms. When the group is $\mathbb{Z}^{d}$, we employ the standard ordering defined by the positive cone $(\mathbb{Z}^{d})^{+}:=(\mathbb{Z}_{\geq0})^{d}$, and whose order unit is $(1,1,\dots,1)$. Recall that $(K_{0}(\mathbb{M}_{d}),K_{0}(\mathbb{M}_{d})^{+},[1])\cong (\mathbb{Z},\mathbb{Z}^{+},d)$, and if $X$ is a zero-dimensional compact metric space, $K_{0}(C(X))\cong C(X;\mathbb{Z})$ with natural point-wise ordering. The $K_{0}$-functor is covariant, namely, if $\phi:A\rightarrow B$ is a $\ast$-homomorphism ($\ast$-automorphism), one obtains a positive group homomorphism (ordered group automorphism) $K_{0}(\phi):K_{0}(A)\rightarrow K_{0}(B)$ defined by $K_{0}(\phi)([p])=[\phi(p)]$ where $p$ is a projection living in $\mathbb{M}_{n}(A)$ for some $n$. For economy we sometimes write $\hat\phi=K_{0}(\phi)$. Note that for every action $\alpha:\Gamma \rightarrow\Aut(A)$, there is an associated action $\hat\alpha:\Gamma\rightarrow\OAut(K_{0}(A))$ where $\hat\alpha(s)= \hat\alpha_{s}:K_{0}(A)\rightarrow K_{0}(A)$ is the induced automorphism. Again, in the case of stable finiteness, the positive cone $K_{0}(A)^{+}$ is a partially ordered monoid, whose ordering is inherited from $K_{0}(A)^{+}$ and coincides with the algebraic ordering. Restricting $\hat\alpha$ to $K_{0}(A)^{+}$ also gives an action of order isomorphisms.

In 1997 Blackadar and Kirchberg introduced in~\cite{BK} the so called \emph{Matricial Field (MF)} algebras. A separable $C^{*}$-algebra $A$ is said to be MF if it can be expressed as a generalized inductive system of finite-dimensional algebras, or equivalently, if there is a natural sequence $\textbf{n}=(n_{k})_{k\geq1}$ and a $*$-monomorphism
\[\iota:A\hookrightarrow Q_{\textbf{n}}:=\prod_{k=1}^{\infty}\mathbb{M}_{n_{k}}\bigg/\bigoplus_{k=1}^{\infty}\mathbb{M}_{n_{k}}.\]
Denote by $\pi:\prod_{k=1}^{\infty}\mathbb{M}_{n_{k}}\rightarrow Q_{\textbf{n}}$ the canonical quotient mapping. If such an embedding $\iota$ exists along with a u.c.p.\ lift, that is, a unital completely positive map $\Phi:A\rightarrow\prod_{k=1}^{\infty}\mathbb{M}_{n_{k}}$ such that $\pi\circ\Phi=\iota$, $A$ is said to be \emph{quasidiagonal}. A good treatise on QD algebras can be found in~\cite{B2}. It is readily seen that an algebra $A$ is MF (QD) if it satisfies the following local property: for every $\varepsilon>0$ and finite set $\Omega\subset A$, there is a $k$ and $*$-linear (u.c.p.) map $\psi:A\rightarrow\mathbb{M}_{k}$ such that
\begin{align*}\|\psi(ab)-\psi(a)\psi(b)\|&<\varepsilon\qquad\forall a,b\in\Omega,\\
\big|\|\psi(a)\|-\|a\|\big|&<\varepsilon\qquad\forall a\in\Omega.
\end{align*}
Recall that a separable algebra $A$ is said to be \emph{residually finite dimensional} (RFD) if there is a sequence of $\ast$-homomorphisms $\psi_{n}:A\rightarrow\mathbb{M}_{k_{n}}$ with $\|\psi_{n}(a)\|\nearrow\|a\|$ for all $a\in A$.
Clearly being MF, QD or RFD passes to \cstar-subalgebras, RFD algebras are QD, and QD algebras are MF. Moreover, MF algebras are stably finite. To see this, suppose $(a_{k})_{k\geq1}\in\prod_{k=1}^{\infty}\mathbb{M}_{n_{k}}$ is a sequence with
\[1_{Q_{\textbf{n}}}=\pi((a_{k})_{k})^{*}\pi((a_{k})_{k})=\pi((a_{k}^{*}a_{k})_{k}),\]
it follows then that $\|a_{k}^{*}a_{k}-1_{\mathbb{M}_{n_{k}}}\|\rightarrow 0$. A little spectral theory shows that $\|a_{k}a_{k}^{*}-1_{\mathbb{M}_{n_{k}}}\|\rightarrow 0$, so that $\pi((a_{k}))\pi((a_{k}))^{*}=1_{Q_{\textbf{n}}}$ thus $Q_{\textbf{n}}$ is finite and hence $A$, being isomorphic to a unital subalgebra of $Q_{\textbf{n}}$, is also finite. Since $M_{n}(A)$ is also MF, $A$ is stably finite. It is still unknown whether or not stably finite algebras are MF, or if there is a countable discrete group $\Gamma$ for which $C^{*}_{\lambda}(\Gamma)$ fails to be MF.

\section{Residually finite and MF actions}

\begin{defn}Let $(X,d)$ be a compact metric space and $\Gamma$ a discrete group. A continuous action $h:\Gamma\rightarrow\Homeo(X)$ is said to be \emph{residually finite} (RF) if for every $\varepsilon>0 $ and finite set $F\subset\Gamma $, there exists a finite set $E$ which admits an action $k:\Gamma\rightarrow\Perm(E)$ and a map $\zeta:E\rightarrow X $ such that
\begin{enumerate}
\item $d(\zeta(k_{s}(z)),h_{s}(\zeta(z)))<\varepsilon$ for each $s\in F$ and $z\in E$,
\item $X\subset_{\varepsilon}\zeta(E)$, that is $\zeta$ has $\varepsilon$-dense range in $X$.
\end{enumerate}
\end{defn}

This notion of a residually finite action was introduced in~\cite{KN}, from which we mention a few results. It is easily verified that if a group admits a \emph{free} residually finite action on some compact space then the group itself must be residually finite, hence the name. Moreover, a residually finite action $\Gamma\curvearrowright X$ will yield a $\Gamma$-invariant probability measure on $X$ which extends in a canonical way (by composition with the conditional expectation) to a trace on $C(X)\rtimes_{\lambda}\Gamma$. Thus residually finite transformation groups $(X,\Gamma)$ produce stably finite reduced crossed products. Theorem~\ref{Main2} below and Lemma 3.9 in~\cite{KN} together show that the converse holds true when $\dim(X)=0$ and $\Gamma=\mathbb{F}_{r}$. It is also shown that when dealing with $\mathbb{Z}$-systems $(X,\mathbb{Z})$, residual finiteness is equivalent to chain recurrence~\cite{C}.

We introduce here a stronger notion than residual finiteness; one that demands exact and global equivariance.

\begin{defn} Let $(X,d)$ be a compact metric space and $\Gamma$ a discrete group. A continuous action $h:\Gamma\rightarrow\Homeo(X)$ is said to be \emph{residually finite dimensional} (RFD) if for every $\varepsilon>0$ there exists a finite set $E$ which admits an action $k:\Gamma\rightarrow\Perm(E)$ and a map $\zeta:E\rightarrow X $ such that
\begin{enumerate}
\item $\zeta(k_{s}(z))=h_{s}(\zeta(z))$ for every $z\in E$ and $s\in\Gamma$.
\item $X\subset_{\varepsilon}\zeta(E)$, that is, $\zeta$ has $\varepsilon$-dense range in $X$.
\end{enumerate}
\end{defn}

In other words, a transformation group $(X,\Gamma)$ is RFD if for every $\varepsilon>0$ there is finite $\Gamma$-invariant subsystem which is $\varepsilon$-dense. Clearly every RFD action is RF, but the converse is false in general; minimal Cantor systems $\mathbb{Z}\curvearrowright X$ yield infinite-dimensional, simple, stably-finite crossed products $C(X)\rtimes\mathbb{Z}$ (\cite{Pu}). As remarked above such systems are residually finite but cannot be residually finite dimensional by Theorem~\ref{KerrNowak} below. The nomenclature is justified by Theorem~\ref{KerrNowak} and Proposition~\ref{RFDactions}.

As observed by the authors of~\cite{KN}, residually finite actions $\Gamma\curvearrowright X$  have \cstar-dynamical expressions when looking at the induced action on the algebra $C(X)$ (see Proposition~\ref{RFactionimpliesQDaction} below). Indeed, what is witnessed at the algebraic level is a finite dimensional approximating property familiar to \cstar-enthusiasts along with an approximate equivariance. We make similar observations when studying RFD actions (see Proposition~\ref{RFDactions}). Here are the appropriate definitions at the \cstar-level.

\begin{defn}\label{defnMFaction} Let $\Gamma$ be a discrete group and $A$ a \cstar-algebra.
\begin{enumerate}
\item An action $\alpha:\Gamma\rightarrow\Aut(A)$ is said to be \emph{matricial field} (MF) provided that: given $\varepsilon > 0$, and finite subsets $ \Omega \subset A$ and $ F \subset \Gamma $, there exist $d\in  \mathbb{N}$, a map $v:\Gamma\rightarrow\mathcal{U}(\mathbb{M}_{d})$ ($s\mapsto v_{s}$), and a unital $*$-linear map $\varphi:A \rightarrow\mathbb{M}_{d}$, such that for every $a,b\in\Omega$ and $s,t\in F$
\begin{enumerate}
\item $\|\varphi(ab)-\varphi(a)\varphi(b)\|<\varepsilon$,
\item $\big|\|\varphi (a)\|-\|a\|\big|<\varepsilon$,
\item $\|\varphi(\alpha_{s}(a))-\Ad_{v_{s}}(\varphi(a))\|< \varepsilon$,
\item $\|v_{st}-v_{s}v_{t}\|<\varepsilon$.
\end{enumerate}
If the unital map $\varphi$ can be further chosen to be completely positive, $\alpha$ is said to be \emph{quasidiagonal} (QD).

\item The action $\alpha:\Gamma\rightarrow\Aut(A)$ is said to be \emph{residually finite dimensional} (RFD) if for every $\varepsilon>0$ and finite subset $\Omega\subset A$, there is a $d\in\mathbb{N}$, a $\ast$-homomorphism $\pi:A\rightarrow\mathbb{M}_{d}$ and a unitary representation $v:\Gamma\rightarrow\mathcal{U}(\mathbb{M}_{d})$ such that
\begin{enumerate}
\item $\|\pi(b)\|>\|b\|-\varepsilon$ for every $b\in\Omega$,
\item $\pi(\alpha_{s}(a))=\Ad_{v(s)}(\pi(a))$ for every $a\in A$ and $s\in\Gamma$.
\end{enumerate}
\end{enumerate}
\end{defn}

A few remarks and key observations concerning Definition~\ref{defnMFaction} are in order. Every RFD system $(A,\Gamma)$ is clearly QD, and every QD system is MF. We show that MF actions are in fact QD when the underlying algebra is amenable (see Proposition~\ref{QDactionequalsMFaction}). It is obvious that if $\alpha:\Gamma\curvearrowright A$ is RFD (QD, MF), then $A$ is itself RFD (QD, MF). Note that any finite dimensional algebra can be embedded into a full matrix algebra, so we may replace $\mathbb{M}_{d}$ by any finite dimensional algebra $B$ without changing the notion. Also, when verifying that an action is MF or QD, it suffices to consider finite subsets of a generating set of the acting group $\Gamma$.

The properties of being MF, QD, or RFD pass to subalgebras, so if a \cstar-system $(A,\Gamma, \alpha)$ yields an MF (QD, RFD) crossed product, one expects the underlying algebra $A$ as well as the group algebra $C^{*}_{\lambda}(\Gamma)$ to be MF (QD, RFD). In these cases one can also decipher information about the action $\alpha$. Indeed, the structure of the reduced crossed product algebra determines the nature of the action.

\begin{prop}\label{MF cross implies MF act} Let $A$ be a unital  \cstar-algebra, and $\alpha:\Gamma\rightarrow\Aut(A)$ a homomorphism. The following hold.
\begin{enumerate}
\item If $A\rtimes_{\lambda,\alpha}\Gamma$ is RFD, then $\cstar_{\lambda}(\Gamma)$ is RFD and the action $\alpha$ is residually finite dimensional.
\item If $A\rtimes_{\lambda,\alpha}\Gamma$ is QD, then $\cstar_{\lambda}(\Gamma)$ is QD and the action $\alpha$ is quasidiagonal.
\item If $A\rtimes_{\lambda,\alpha}\Gamma$ is MF, then $\cstar_{\lambda}(\Gamma)$ is MF and the action $\alpha$ is matricial field. Moreover, if $A$ is nuclear, $\alpha$ is quasidiagonal.
\end{enumerate}
\end{prop}

\begin{proof} Suppose $A\rtimes_{\lambda,\alpha}\Gamma$ is residually finite dimensional. Again, being RFD passes to subalgebras, so $\cstar_{\lambda}(\Gamma)$ is RFD as it sits canonically inside the reduced crossed product. If $\varepsilon>0$ and if $\Omega\subset A$ is a finite set, then there is a $d$ and a $\ast$-homomorphism $\phi:A\rtimes_{\lambda,\alpha}\Gamma\rightarrow\mathbb{M}_{d}$ such that $\|\phi(\iota(b))\|>\|\iota(b)\|-\varepsilon=\|b\|-\varepsilon$ for every $b\in\Omega$, where $\iota:A\hookrightarrow A\rtimes_{\lambda,\alpha}\Gamma$ denotes the natural inclusion. Set $\pi=\phi\circ\iota:A\rightarrow\mathbb{M}_{d}$. Now define a unitary representation $v:\Gamma\rightarrow\mathcal{U}(d)$ as $v(s)=\phi(1_{A}u_{s})$, where $u_{s}$ denote the canonical unitaries in the crossed product implementing the action. Set $\gamma_{s}=\Ad_{v(s)}$ so that $\gamma:\Gamma\curvearrowright\mathbb{M}_{d}$ is an action. We verify
\begin{align*}\pi(\alpha_{s}(a))&=\phi(\iota(\alpha_{s}(a)))=\phi(\alpha_{s}(a)u_{e})=\phi(u_{s}au_{e}u_{s}^{*})
=\phi(u_{s})\phi(\iota(a))\phi(u_{s})^{*}\\&=v(s)\pi(a)v(s)^*=\gamma_{s}(\pi(a))
\end{align*}
and this completes the proof of (1).

We prove (3) next. Let $\varepsilon>0$, and let $F\subset\Gamma$, $\mathcal{F}\subset A$ be finite subsets. For $s\in \Gamma$ denote by $u_{s}$ the canonical unitaries in $A\rtimes_{\lambda,\alpha}\Gamma$ and for economy write $B=A\rtimes_{\lambda,\alpha}\Gamma$. Now set
\[\Omega=\{u_{s}, u_{s}^{*}, u_{st}, u_{s}a, a\ :\ s,t\in F,\ a\in\mathcal{F}\}\subset B.\]
Next, set $K=\max_{x\in\Omega}\|x\|+1$. By an elementary perturbation result, there is a $0<\delta<\min\{1,\varepsilon/4\}$ with the following property: if $D$ is \emph{any} \cstar-algebra with $d\in D$ satisfying $\|d^{*}d-1\|<\delta$ and $\|dd^{*}-1\|<\delta$, there is a $v\in\mathcal{U}(D)$ with $\|d-v\|<\varepsilon/4K$. Assuming $B$ is MF, there is a unital $\ast$-linear map $\psi:B\rightarrow \mathbb{M}_{d}$ such that
\begin{align}
\label{AMpsi}\|\psi(xy)-\psi(x)\psi(y)\|&<\delta \qquad \forall x,y,\in\Omega,\\
\label{AIpsi}|\|\psi(x)\|-\|x\||&<\delta \qquad \forall x\in \Omega.
\end{align}

Since $\delta<\varepsilon$, $\mathcal{F}\subset\Omega$ and in light of the inequalities~\eqref{AMpsi} and ~\eqref{AIpsi}, all what is needed to show now is approximate equivariance with an appropriate map $v:\Gamma\rightarrow\mathcal{U}(\mathbb{M}_{d})$. Note that for $r\in F$ or $r=st$ with $s,t\in F$ we have
\begin{align*}\|\psi(u_{r})^{*}\psi(u_{r})-1\|&=\|\psi(u_{r}^{*})\psi(u_{r})-\psi(u_{r}^{*}u_{r})\|<\delta,\\
\|\psi(u_{r})\psi(u_{r})^{*}-1\|&=\|\psi(u_{r})\psi(u_{r}^{*})-\psi(u_{r}u_{r}^{*})\|<\delta,
\end{align*}
therefore, by our choice of $\delta$, there are unitaries $v_{r}$, in $\mathbb{M}_{d}$ for each $r\in F$ or $r=st$ with $s,t\in F$ that satisfy $\|v_{r}-\psi(u_{r})\|<\varepsilon/4K$. Extend $v:\Gamma\rightarrow\mathcal{U}(\mathbb{M}_{d})$ arbitrarily. We need to show that for every $a$ in $\mathcal{F}$ and $s,t\in F$
\begin{gather}\label{AE}
\|\psi(\alpha_{s}(a))-\Ad_{v_{s}}(\psi(a))\|=\|\psi(u_{j}au_{j}^{*})-v_{j}\psi(a)v_{j}^{*}\|<\varepsilon,\\
\label{un}\|v_{st}-v_{s}v_{t}\|<\varepsilon.
\end{gather}
To this end, first note that~\eqref{AIpsi} implies that for each $x\in\Omega$
\[\|\psi(x)\|=\|\psi(x)\|-\|x\|+\|x\|\leq \big|\|\psi(x)-\|x\|\big|+\|x\|< \delta+(K-1)<K.\]
Using \eqref{AMpsi} along with the definition of the $v_{s}$ yields
\begin{align*}\|\psi(u_{s}au_{s}^{*})-v_{s}\psi(a)v_{s}^{*}\|&\leq \|\psi(u_{s}au_{s}^{*})-\psi(u_{s}a)\psi(u_{s}^{*})\|+\|\psi(u_{s}a)\psi(u_{s}^{*})-\psi(u_{s}a)v_{s}^{*}\|\\
&\qquad +\|\psi(u_{s}a)v_{s}^{*}-\psi(u_{s})\psi(a)v_{s}^{*}\|+\|\psi(u_{s})\psi(a)v_{s}^{*}-v_{s}\psi(a)v_{s}^{*}\|\\
&<\delta + \|\psi(u_{s}a)\|\|\psi(u_{s})^{*}-v_{s}^{*}\|\\
&\qquad+\|\psi(u_{s}a)-\psi(u_{s})\psi(a)\|\|v_{s}^{*}\|+\|\psi(u_{s})-v_{s}\|\|\psi(a)v_{s}^{*}\|\\
&<\delta +K\cdot\frac{\varepsilon}{4K}+\delta+\frac{\varepsilon}{4K}\cdot K<4\cdot\frac{\varepsilon}{4}=\varepsilon,
\end{align*}
which establishes~\eqref{AE}. To see~\eqref{un},
\begin{align*}\|v_{st}-v_{s}v_{t}\|&\leq\|v_{st}-\psi(u_{st})\|+\|\psi(u_{s}u_{t})-\psi(u_{s})\psi(u_{t})\|+
\|\psi(u_{s})\psi(u_{t})-v_{s}\psi(u_{t})\|+\|v_{s}\psi(u_{t})-v_{s}v_{t}\|\\ &\leq \frac{\varepsilon}{4K}+\delta+\|\psi(u_{t})\|\frac{\varepsilon}{4K}+\frac{\varepsilon}{4K}<\varepsilon.
\end{align*}
The action is thus MF. If $A$ is amenable, Proposition~\ref{QDactionequalsMFaction} ensures that $\alpha$ is QD  and the proof of (3) is complete. The proof of (2) is identical except for the fact that we may choose $\psi$ to be completely positive provided that $B$ is quasidiagonal.
\end{proof}

We now embark on establishing partial converses to Proposition~\ref{MF cross implies MF act}. Reduced crossed products emerging from a QD \cstar-system exhibit a finite-dimensional approximating property, they admit norm microstates. Indeed, it was shown in~\cite{KN}, in the separable case, that if $\alpha$ is quasidiagonal and $\cstar_{\lambda}(\Gamma)$ is MF, the reduced crossed product algebra $A\rtimes_{\lambda,\alpha}\Gamma$ is also MF. However, we must point out that the definition of a QD action in~\cite{KN} is somewhat stronger than Defintion~\ref{defnMFaction}. In lieu of a local approximately multiplicative map $v:\Gamma\rightarrow\mathcal{U}(\mathbb{M}_{d})$, the authors require a legitimate action $\gamma:\Gamma\curvearrowright\mathbb{M}_{d}$  with
\begin{enumerate}
\item[(c$'$)] $\|\varphi(\alpha_{s}(a))-\gamma_{s}(\varphi(a))\|<\varepsilon$, for every $a\in\Omega$ and $s\in F$.
\end{enumerate}
Assuming that such an action $\gamma$ exists, apply the GNS construction to $(\mathbb{M}_{d},\tau)$, where $\tau$ is the unique faithful tracial state on $\mathbb{M}_{d}$, to obtain the faithful representation $\pi_{\tau}:\mathbb{M}_{d}\rightarrow\mathbb{B}(L^{2}(\mathbb{M}_{d},\tau))\cong\mathbb{M}_{d^{2}}$. Then define unitaries in $\mathbb{M}_{d^{2}}$ by $v_{s}(\hat x)=\widehat{\gamma_{s}(x)}$ for $s\in\Gamma$ and $x\in\mathbb{M}_{d}$. It is easily verified that $v:\Gamma\rightarrow\mathcal{U}(\mathbb{M}_{d^2})$ is in fact a unitary representation satisfying $v_{s}\pi_{\tau}(x)v_{s}^{*}=\pi_{\tau}(\gamma_{s}(x))$ for $x\in\mathbb{M}_{d}$. Replacing $\varphi$ by $\pi_{\tau}\circ\varphi$ and $\mathbb{M}_{d}$ by $\mathbb{M}_{d^2}$ we then have an MF (or QD) action in the above sense of~\ref{defnMFaction}. Therefore Definition~\ref{defnMFaction} is a weakening of that given in~\cite{KN}. With some extra work one can still prove Theorem 3.4 in~\cite{KN} with our weakened definition of a QD action. We include this result for completeness along with other partial converses to Proposition~\ref{MF cross implies MF act}.

\begin{thm}\label{KerrNowak} Let $\Gamma$ be a discrete group and $\alpha:\Gamma\rightarrow\Aut(A)$ an action on a separable unital \cstar-algebra $A$.
\begin{enumerate}
\item The reduced crossed product $A\rtimes_{\lambda,\alpha}\Gamma$ is RFD if and only if $\cstar_{\lambda}(\Gamma)$ is RFD and $\alpha$ is RFD.
\item If $\cstar_{\lambda}(\Gamma)$ is MF and $\alpha$ is QD, then the reduced crossed product $A\rtimes_{\lambda,\alpha}\Gamma$ is MF.
\item If $\cstar_{\lambda}(\Gamma)$ is QD and $\alpha$ is MF, then the reduced crossed product $A\rtimes_{\lambda,\alpha}\Gamma$ is MF.
\end{enumerate}
\end{thm}

\begin{proof} (1): Consider an RFD action $\alpha:\Gamma\curvearrowright A$. We then have a sequence $\ast$-homomorphisms $\pi_{n}:A\rightarrow\mathbb{M}_{k_{n}}$ and unitary representations $v_{n}:\Gamma\rightarrow\mathcal{U}(\mathbb{M}_{k_{n}})$ such that for every $s\in\Gamma$ and $a\in A$:
\begin{enumerate}
\item[(i)] $\|\pi_{n}(a)\|\nearrow\|a\|$ as $n\rightarrow\infty$.
\item[(ii)] $\pi_{n}(\alpha_{s}(a))=\Ad_{v_{n}(s)}(\pi_{n}(a))$.
\end{enumerate}

For economy write $M=\prod_{n=1}^{\infty}\mathbb{M}_{k_{n}}$, and $\mathcal{U}(k_{n})= \mathcal{U}(\mathbb{M}_{k_{n}})$. Consider the unitary representation $v:\Gamma\rightarrow\mathcal{U}(M)\cong\prod_{n\geq1}\mathcal{U}(k_{n})$ given by $v(s):=(v_{n}(s))_{n\geq1}$, and the $\ast$-homomorphism $\pi:A\rightarrow M$ defined by $\pi(a):=(\pi_{n}(a))_{n\geq1}$. Also set $\beta_{s}=\Ad_{v(s)}$, so that $\beta:\Gamma\rightarrow\Aut(M)$ is an action. Condition (i) ensures that $\pi$ is injective, and condition (ii) implies equivariance of $\pi$, that is $\pi(\alpha_{s}(a))=\beta_{s}(\pi(a))$ for every $s\in\Gamma$ and $a\in A$. 
We thus have an monomorphism of \cstar-dynamical systems $\pi:(A,\Gamma,\alpha)\rightarrow(M,\Gamma,\beta)$. Therefore, $A\rtimes_{\lambda,\alpha}\Gamma\hookrightarrow M\rtimes_{\lambda,\beta}\Gamma$. Since $\beta$ is an inner action, we know that $M\rtimes_{\lambda,\beta}\Gamma\cong M\otimes_{\min}\cstar_{\lambda}(\Gamma)$, whence
\[A\rtimes_{\lambda,\alpha}\Gamma\hookrightarrow M\otimes_{\min}\cstar_{\lambda}(\Gamma).\]

Now simply note that since both $M$ and $\cstar_{\lambda}(\Gamma)$ are RFD algebras, so is their  minimal tensor product $M\otimes_{\min}\cstar_{\lambda}(\Gamma)$. Being RFD passes to subalgebras, so we conclude $A\rtimes_{\lambda,\alpha}\Gamma$ is RFD.

(2): (The proof of this is essentially the same proof as Theorem 3.4 in~\cite{KN}.)

(3): We have a sequence of  $\ast$-linear maps $\varphi_{n}:A\rightarrow\mathbb{M}_{k_{n}}$ and maps $v_{n}:\Gamma\rightarrow\mathcal{U}(\mathbb{M}_{k_{n}})$ such that for every $s,t\in\Gamma$ and $a,b\in A$:
\begin{enumerate}
\item[(i)] $\|\varphi_{n}(ab)-\varphi_{n}(a)\varphi_{n}(b)\|\rightarrow 0$, as $n\rightarrow\infty$.
\item[(ii)]$\|\varphi_{n}(a)\|\rightarrow\|a\|$ as $n\rightarrow\infty$.
\item[(iii)]$\|\varphi_{n}(\alpha_{s}(a))-\Ad_{v_{n}(s)}(\varphi_{n}(a))\|\rightarrow 0$, as $n\rightarrow\infty$.
\item[(iv)]$\|v_{n}(st)-v_{n}(s)v_{n}(t)\|\rightarrow0$, as $n\rightarrow\infty$.
\end{enumerate}

Write $M=\prod_{n=1}^{\infty}\mathbb{M}_{k_{n}}/\bigoplus_{n=1}^{\infty}\mathbb{M}_{k_{n}}$ and $\pi:\prod_{n=1}^{\infty}\mathbb{M}_{k_{n}}\rightarrow M$ the canonical quotient map. Now consider the maps $\Phi:A\rightarrow M$ and $v:\Gamma\rightarrow\mathcal{U}(M)$ given by
\[\Phi(a):=\pi[(\phi_{n}(a))_{n}]\qquad v(s):=\pi[(v_{n}(s))_{n}].\]
By properties (i) and (ii) $\Phi$ is an $*$-monomorphism, and by property (iv) $v$ is a unitary representation. Therefore, we have an inner action $\beta:\Gamma\curvearrowright M$ defined by $\beta_{s}:=\Ad_{v_{s}}$ for $s\in\Gamma$. Note that property (iii) implies that $\beta_{s}(\Phi(a))=\Phi(\alpha_{s}(a))$ for every $a\in A$ and $s\in\Gamma$. Indeed
\begin{align*}\beta_{s}(\Phi(a))&=\Ad_{v_{s}}(\Phi(a))=\pi[(v_{n}(s))_{n}]\pi[(\phi_{n}(a))_{n}]\pi[(v_{n}(s))_{n}]^{*}
=\pi[(v_{n}(s)\phi_{n}(a)v_{n}(s)^{*})_{n}]\\&=\pi[(\phi_{n}(\alpha_{s}(a)))_{n}]=\Phi(\alpha_{s}(a)).
\end{align*}

We thus have a monomorphism of \cstar-dynamical systems $\Phi:(A,\Gamma,\alpha)\rightarrow(M,\Gamma,\beta)$. Therefore, $A\rtimes_{\lambda,\alpha}\Gamma\hookrightarrow M\rtimes_{\lambda,\beta}\Gamma$. Since $\beta$ is an inner action, by Example 9.11 of~\cite{Ph} we know that $M\rtimes_{\lambda,\beta}\Gamma\cong M\otimes_{\min}\cstar_{\lambda}(\Gamma)$, whence
\[A\rtimes_{\lambda,\alpha}\Gamma\hookrightarrow M\otimes_{\min}\cstar_{\lambda}(\Gamma).\]

Since $\cstar_{\lambda}(\Gamma)$ is QD, $\Gamma$ is amenable, and so $\cstar_{\lambda}(\Gamma)$ is nuclear. It follows from Proposition 3.3.6 of~\cite{BK} that  $M\otimes_{\min}\cstar_{\lambda}(\Gamma)$ is MF, which implies that $A\rtimes_{\lambda,\alpha}\Gamma$ is MF.
\end{proof}

The following is Proposition 3.3 in~\cite{KN} and justifies the comments made prior to defining MF actions.

\begin{prop}\label{RFactionimpliesQDaction}Let $(X,d)$ be a compact metric space and $h:\Gamma\rightarrow\Homeo(X)$ a continuous action with induced action $\alpha$ on $C(X)$. If $h$ is residually finite, then $\alpha$ is quasidiagonal.
\end{prop}

The question emerges of whether the converse to the previous result holds. The authors of~\cite{KN} showed that in the case of an action $h:\mathbb{F}_{r}\curvearrowright X$ on a compact zero-dimensional metric space, $h$ is residually finite if and only if the induced action on $C(X)$ is quasidiagonal.

In the same vein we relate RFD actions on compact metric spaces with RFD actions at the algebraic level.

\begin{prop}\label{RFDactions} Let $h:\Gamma\rightarrow\Homeo(X)$ be a continuous action on a compact metric space $(X,d)$, and $\alpha:\Gamma\rightarrow\Aut(C(X))$ the induced action. Then $h$ is RFD if and only if $\alpha$ is RFD.
\end{prop}

\begin{proof} Consider first a RFD transformation group $(X,\Gamma)$. Let $g\in C(X)$ and $\varepsilon>0$ be given. By compactness there is a $\delta>0$ such that
\[x,y\in X,\quad d(x,y)<\delta\Longrightarrow |g(x)-g(y)|<\varepsilon.\]
We then obtain a finite set $E$, an action $\Gamma\curvearrowright E$ and a map $\zeta:E\rightarrow X$ with $\zeta(E)\subset_{\delta}X$ and $\zeta(s.z)=s.\zeta(z)$ for every $z\in E$ and $s\in\Gamma$. Dualize by defining $\pi:C(X)\rightarrow C(E)$ as $\pi(f)=f\circ\zeta$ for $f\in C(X)$ and $\gamma_{s}(k)(z):=k(s.^{-1}z)$ for $k\in C(E)$, $s\in\Gamma$ and $z\in E$. The equivariance is straightforward, indeed for $f\in C(X)$, $s\in\Gamma$, $z\in E$ we have
\[\pi(\alpha_{s}(f))(z)=\alpha_{s}(f)(\zeta(z))=f(s^{-1}.(\zeta(z))=f(\zeta(s^{-1}.z))=\gamma_{s}(f\circ\zeta)(z)
=\gamma_{s}(\pi(f))(z)\]
which implies $\pi(\alpha_{s}(f))=\gamma_{s}(\pi(f))$ for every $f\in C(X)$ and $s\in\Gamma$.

For the approximate isometry condition, fix $x$ in $X$, and pick up a $z_{x}\in E$ with $d(x,\zeta(z_{x}))<\delta$. Then observe
\[|g(x)|\leq |g(x)-g(\zeta(z_{x}))|+|g(\zeta(z_{x}))|<\varepsilon+\sup_{z\in E}|g\circ\zeta(z)|=\varepsilon+\|\pi(g)\|.\]
Taking a supremum over all $x\in X$ gives $\|g\|\leq \varepsilon+\|\pi(g)\|$.

Conversely, suppose that $\alpha$ is an RFD action. We then have a sequence of finite dimensional representations $\pi_{n}:C(X)\rightarrow\mathbb{M}_{k_{n}}$ and actions $\gamma_{n}:\Gamma\curvearrowright\mathbb{M}_{k_{n}}$ such that \begin{enumerate}
\item[(i)] $\|\pi_{n}(f)\|\nearrow\|f\|$ for each $f\in C(X)$
\item[(ii)] $\pi_{n}(\alpha_{s}(f))=\gamma_{n,s}(\pi_{n}(f))$ for every $s\in\Gamma$, $f\in C(X)$ and $n\geq1$.
\end{enumerate}

Fix an $n\geq 1$ and note that $\pi_{n}(C(X))$ is a finite dimensional commutative algebra, therefore isomorphic to $C(E_{n})$ for some finite set $E_{n}$. Also note that $\pi_{n}(C(X))$ is invariant under the action $\gamma_{n}$ by condition $(ii)$, so by restricting, we may suppose $\Gamma$ acts on $C(E_{n})$ via $\gamma_{n}$. We therefore have maps $\zeta_{n}:E_{n}\rightarrow X$ with $\pi_{n}(f)=f\circ\zeta_{n}$ for each $f\in C(X)$, and actions $\Gamma\curvearrowright E_{n}$ implemented by the homomorphisms $\gamma_{n}$. We verify the promised equivariance. Indeed, for $f\in C(X)$, $z\in E_{n}$ and $s\in\Gamma$ we have
\[f(s^{-1}.\zeta_{n}(z))=\alpha_{s}(f)(\zeta_{n}(z))=\pi_{n}(\alpha_{s}(f))(z)=\gamma_{n,s}(\pi_{n}(f))(z)
=\pi_{n}(f)(s^{-1}.z)=f(\zeta_{n}(s^{-1}.z)).\]
Recall that $C(X)$ separates points of $X$ so that $s^{-1}.\zeta_{n}(z)=\zeta_{n}(s^{-1}.z)$ for all $s\in\Gamma$ and $z\in E$, and equivariance follows.

Consider now an arbitrary $\varepsilon>0$. We claim that for some $n$ large we have $X\subset_{\varepsilon}\zeta_{n}(E_{n})$. Suppose not. Then for every $n$ there is an $x_{n}\in X$ with $d(\zeta_{n}(z),x_{n})\geq\varepsilon$ for every $z\in E_{n}$. Passing to a subsequence we may assume that $(x_{n})_{n}$ converges to some $x_{0}\in X$. Note that for some $N$ large we have that for every $n\geq N$ and $z\in E_{n}$
\[\varepsilon\leq d(\zeta_{n}(z),x_{n})\leq d(\zeta_{n}(z),x_{0})+d(x_{0},x_{n})\leq d(\zeta_{n}(z),x_{0})+\varepsilon/2,\]
thus $d(\zeta_{n}(z),x_{0})\geq \varepsilon/2$ holds for every such $n$ and $z\in E_{n}$. Now choose a continuous $f:X\rightarrow[0,1]$ with $f(x_{0})=1$ and $\supp(f)\subset\overline{B(x_{0},\varepsilon/3)}$. Thus $\|f\|=1$, but $\pi_{n}(f)=f\circ\zeta_{n}=0$ since $\zeta_{z}(E_{n})\subset \supp(f)^{c}$. This contradicts condition $(i)$, so the Claim holds and the proof is complete.
\end{proof}

We now want to look at some examples of MF actions. The first class of \cstar-systems seems tailored to admit finite-dimensional approximating dynamics.

\begin{example} For a fixed discrete group $\Gamma$, let $(A_{n},\Gamma,\alpha^{(n)})_{n\geq1}$ be a sequence of \cstar-dynamical systems with each $A_{n}$ finite dimensional. By standard inductive limit techniques one constructs the \cstar-system $(A,\Gamma,\alpha)$ where $A:=\bigotimes_{n\geq1}A_{n}$ and $\alpha:=\otimes_{n\geq1}\alpha^{(n)}$ acts via
\[\alpha_{s}(a_{n_{1}}\otimes\cdots\otimes a_{n_{k}})=\alpha_{s}^{(n_{1})}(a_{n_{1}})\otimes\cdots\otimes \alpha_{s}^{(n_{k})}(a_{n_{k}})\qquad s\in\Gamma.\]
Given a finite subset $\Omega\subset A$, one can find a large enough $m$ and approximate each member of $\Omega$ by elements from the subalgebra $B_{m}:=\bigotimes_{n=1}^{m}A_{n}$. The identity map on $B_{m}$ lifts to a u.c.p map $\varphi:A\rightarrow B_{m}$, and $\Gamma$ acts on $B_{m}$ as $\beta^{(m)}=\alpha^{(1)}\otimes\cdots\otimes\alpha^{(m)}$. The conditions for a QD action are now easily verified.
\end{example}

More instances of QD actions will surface as we uncover their theory, but we can immediately provide a wide class of examples. Recall that for a unital \cstar-algebra $A$,
\[\Inn(A)=\{\Ad_{u} : u\in\mathcal{U}(A)\}\leq\Aut(A)\]
denotes the normal subgroup of inner automorphisms, while $\overline{\Inn}(A)\leq\Aut(A)$ is the normal subgroup of all approximately inner automorphisms. Given $\varepsilon>0$, a finite set $\mathcal{F}\subset A$ and $\alpha\in\overline{\Inn}(A)$, there is an inner automorphism $\Ad_{u}$ with $\|\Ad_{u}(x)-\alpha(x)\|\leq\varepsilon$ for every $x\in\mathcal{F}$.

\begin{prop}Let $A$ be a unital AF algebra with $\overline{\Inn}(A)=\Aut(A)$. Then any action $\alpha:\mathbb{F}_{r}\rightarrow\Aut(A)$ of a free group on $A$ is quasidiagonal. In particular, if $A$ is AF with $K_{0}(A)$ totally ordered and Archimedean, or if $A$ is UHF, then any action of $\mathbb{F}_{r}$ on $A$ is quasidiagonal.
\end{prop}

\begin{proof}Let $0<\varepsilon<1$ and let $\Omega\subset A$  be a finite set. Also, denote the generators of $\mathbb{F}_{r}$ by $s_{1},\dots,s_{r}$. Set $K=\max_{a\in\Omega}\|a\|+1$ and put $\delta=\min\{\varepsilon/(3+2K), \varepsilon/4K, \varepsilon/2\}$. Since $A$ is AF and unital, locate a unital finite-dimensional subalgebra $1_{A}\in B\subset A$ and a finite subset $\Omega'\subset B$ with $\|a-b\|<\delta$ for $a$ in $\Omega$ and $b$ in $\Omega'$.

Since every automorphism on $A$ is approximately inner, there are unitaries $u_{1},\dots, u_{r}\in\mathcal{U}(A)$ such that
\begin{equation}\|u_{j}bu_{j}^{*}-\alpha_{s_{j}}(b)\|<\delta\qquad \forall b\in\Omega',\ \forall j\in\{1,\dots, r\}.
\end{equation}
By standard perturbation results, we can find a unital finite-dimensional algebra $1_{A}\in D\subset A$ along with unitaries $v_{1},\dots,v_{r}$ in $D$ such that $B\subset D$ and $\|u_{j}-v_{j}\|<\delta$ for every $j$. We then have automorphisms of $D$ for each $j=1,\dots,r$, namely $\Ad_{v_{j}}:D\rightarrow D$ given by $\Ad_{v_{j}}(x)=v_{j}xv_{j}^{*}$ for $x\in D$. By the universal property of the free group, the map $\{s_{1},\dots,s_{r}\}\rightarrow\Aut(D)$ where $s_{j}\mapsto \Ad_{v_{j}}$ extends to a group homomorphism
\[\gamma:\mathbb{F}_{r}\rightarrow\Aut(D)\qquad \gamma_{s_{j}}=\Ad_{v_{j}}\quad \forall j\in\{1,\dots,r\}.\]
Appealing to Arveson's Extension Theorem, let $\varphi:A\rightarrow D$ be the u.c.p extension of $\id_{D}:D\rightarrow D$. We work out the necessary estimates. First, since each $a\in\Omega$ is $\delta$-close to a $b\in\Omega'$, we have
\begin{align*}\|\varphi(\alpha_{s_{j}}(a))&-\gamma_{s_{j}}(\varphi(a))\|\\ &\leq\|\varphi(\alpha_{s_{j}}(a))-\varphi(\alpha_{s_{j}}(b))\|+\|\varphi(\alpha_{s_{j}}(b))-\gamma_{s_{j}}(\varphi(b))\|
+\|\gamma_{s_{j}}(\varphi(b))-\gamma_{s_{j}}(\varphi(a))\|\\
&\leq\|a-b\|+\|\varphi(\alpha_{s_{j}}(b))-\gamma_{s_{j}}(\varphi(b))\|+\|a-b\|\\
&\leq 2\delta +\|\varphi(\alpha_{s_{j}}(b))-\gamma_{s_{j}}(\varphi(b))\|.
\end{align*}
Next, we use the fact that $\gamma_{s_{j}}(\varphi(b))=\gamma_{s_{j}}(b)=v_{j}bv_{j}^{*}=\varphi(v_{j}bv_{j}^{*})$ since $\varphi|_{D}=\id_{D}$ and the elements $b$ and $v_{j}bv_{j}^{*}$ all belong to $D$. This together with (1) gives:
\begin{align*}\|\varphi(\alpha_{s_{j}}(b))-\gamma_{s_{j}}(\varphi(b))\|&\leq \|\varphi(\alpha_{s_{j}}(b))-\varphi(u_{j}bu_{j}^{*})\|+\|\varphi(u_{j}bu_{j}^{*})-\varphi(v_{j}bv_{j}^{*})\|\\
&\leq\|\alpha_{s_{j}}(b)-u_{j}bu_{j}^{*}\|+\|u_{j}bu_{j}^{*}-v_{j}bv_{j}^{*}\|\\
&<\delta + \|u_{j}bu_{j}^{*}-v_{j}bv_{j}^{*}\|.
\end{align*}
The unitaries $u_{j}$ and $v_{j}$ are $\delta$-close so we get
\begin{align*}\|u_{j}bu_{j}^{*}-v_{j}bv_{j}^{*}\|&\leq \|u_{j}bu_{j}^{*}-u_{j}bv_{j}^{*}\|+\|u_{j}bv_{j}^{*}-v_{j}bv_{j}^{*}\|=\|u_{j}b(u_{j}^{*}-v_{j}^{*})\|+\|(u_{j}-v_{j})bv_{j}^{*}\|\\
&\leq \|b\|\|u_{j}^{*}-v_{j}^{*}\|+\|u_{j}-v_{j}\|\|b\|\leq 2K\delta
\end{align*}
All of the above estimates yield $\|\varphi(\alpha_{s_{j}}(a))-\gamma_{s_{j}}(\varphi(a))\|<(2K+3)\delta\leq \varepsilon$ for every generator $s_{j}$ and every $a\in\Omega$. This gives the desired approximate equivariance. We still have yet to show that $\varphi$ is approximately isometric and approximately multiplicative on $\Omega$. To that end, let $x,y\in\Omega$ and let $x',y'\in\Omega'$ be their $\delta$-approximations. Note that since $\delta<1$, $\|x\|\leq K-1$ and $\|x-x'\|<\delta$, it easily follows that $\|x'\|\leq K$. A simple triangle inequality gives
\[\|xy-x'y'\|\leq\|xy-x'y\|+\|x'y-x'y'\|\leq\|x-x'\|\|y\|+\|x'\|\|y-y'\|\leq 2\delta K.\]
Similarly $\|\varphi(x)\varphi(y)-\varphi(x')\varphi(y')\|\leq 2\delta K$, since $\varphi$ is contractive. Recalling that $\varphi$  restricted to $D$ is the identity, our above estimates yield
\begin{align*}\|\varphi(xy)-\varphi(x)\varphi(y)\|&\leq \|\varphi(xy)-\varphi(x'y')\|+\|\varphi(x'y')-\varphi(x')\varphi(y')\|+
\|\varphi(x')\varphi(y')-\varphi(x)\varphi(y)\|\\&\leq \|xy-x'y'\|+0+\|\varphi(x)\varphi(y)-\varphi(x')\varphi(y')\|\leq 4\delta K<\varepsilon.
\end{align*}
This gives the approximate multiplicativity. Finally, $\varphi$ is easily seen to be approximately isometric:
\begin{align*}\big|\|\varphi(x)\|-\|x\|\big|&\leq \big|\|\varphi(x)\|-\|\varphi(x')\|\big|+\big|\|\varphi(x')\|-\|x\|\big|\leq
\|\varphi(x)-\varphi(x')\|+\big|\|x'\|-\|x\|\big|\\ &\leq \|\varphi(x-x')\|+\|x-x'\|\leq 2\|x-x'\|\leq 2\delta\leq \varepsilon,
\end{align*}
which confirms that $\alpha$ is indeed quasidiagonal.

If $A$ is a unital AF algebra such that $K_{0}(A)$ is totally ordered and Archimedean then $\overline{\Inn(A)}=\Aut(A)$, which is indeed the case for UHF algebras (see Corollary IV.5.8 in~\cite{Da}).
\end{proof}

The next example of a QD action is a generalization of Voiculescu's notion of an action of $\mathbb{Z}$ admitting $\emph{pseudo-orbits}$ described in~\cite{Vo}. For an algebra $A$, write $\mathcal{F}(A)$ for the collection all of finite-dimensional subalgebras of $A$. Also, if $B,C$ are \cstar-subalgebras of $A$ and $\varepsilon>0$, we shall write $B\subset_{\varepsilon}C$ if $\sup_{b\in \Ball(B)}d(b,\Ball(C))<\varepsilon$, and $d(B,C)$ is defined by
\[d(B,C)=\inf\{\varepsilon>0 : B\subset_{\varepsilon}C\ \mbox{and}\ C\subset_{\varepsilon}B\}.\]

\begin{defn} $(A,\Gamma,\alpha)$ be a \cstar dynamical system. $\alpha$ is said to have the \emph{pseudo-orbit property} if for every $\varepsilon>0$, $F\subset\Gamma$ finite subset and $D\in\mathcal{F}(A)$, there is a finite quotient $\pi:\Gamma\rightarrow\Lambda$ along with a map $\zeta:\Lambda\rightarrow\mathcal{F}(A)$ such that
\begin{enumerate}
\item $D\subset B_{t}:=\zeta(t)$, for every $t\in\Lambda$,
\item $d(\alpha_{s}(B_{t}), B_{\pi(s)t})<\varepsilon$ for every $t\in\Lambda$ and $s\in F$.
\end{enumerate}
\end{defn}

Before stating the Proposition, we remind the reader of a perturbation result due to E. Christensen (see~\cite{Ch}) which reads as following.

\begin{lem} For every $\delta>0$, there is a $\delta_{1}>0$ such that whenever $B$ and $C$ are \cstar-subalgebras of a unital \cstar-algebra $A$ with $B$ finite dimensional and $C\subset_{\delta_{1}}B$, then there is a unitary $u\in A$ with $\|u-1\|<\delta$ and $\Ad_{u}(B)\subset C$.
\end{lem}

\begin{prop} Let $A$ be an AF-algebra, $r\in\mathbb{N}$ and $\alpha:\mathbb{F}_{r}\rightarrow\Aut(A)$ an action with the pseudo-orbit property. Then $\alpha$ is quasidiagonal.
\end{prop}

\begin{proof}

Let $\varepsilon>0$, $\Omega\subset A$ a finite subset and $F=\{e=s_{0},s_{1},\dots, s_{r}\}$, where $s_{1},\dots,s_{r}$ are the standard generators for $\mathbb{F}_{r}$. Let $C:=\max_{x\in\Omega}\|x\|+1$ and let $\delta$ be so small that $2\delta<\varepsilon$, $4C\delta<\varepsilon$ and $2\delta(2+C)<\varepsilon$. Since $A$ is an AF algebra, we may choose a finite dimensional subalgebra $D\subset A$ with $\alpha_{s}(\Omega)\subset_{\delta}D$ for every $s\in F$.  Let $\delta_{1}=\delta_{1}(\delta)>0$ be a perturbation constant as in Christensen's result above. By our hypothesis, there is a finite quotient $\pi:\mathbb{F}_{r}\rightarrow\Lambda$ and a map $\zeta: \Lambda\rightarrow \mathcal{F}(A)$ with $D\subset B_{t}$ and $d(\alpha_{s}(B),B_{\pi(s)t})<\delta_{1}$ for each $t\in\Lambda$ and $s\in F$. Thus for each pair $(s,t)\in F\times\Lambda$, find unitaries $u_{s,t}\in\mathcal{U}(A)$ with $\|u_{s,t}-1\|<\delta$ and $\Ad_{u_{s,t}}(\alpha_{s}(B_{t}))\subset B_{\pi(s)t}$.

Now set $B=\bigoplus_{t\in\Lambda}B_{t}$ and for each $s\in F$ consider the automorphism of $B$ given by
\[\sigma_{s}((b_{t})_{t\in\Lambda})=(u_{s,t}\alpha_{s}(b_{t})u_{s,t}^{*})_{t\in\Lambda}.\]
We thus have an action $\sigma:\mathbb{F}_{r}\rightarrow\Aut(B)$. Let $\varphi:A\rightarrow B$ be the u.c.p.\ extension of the inclusion $D\hookrightarrow B$ given by $a\mapsto(a)_{t\in\Lambda}$. Fixing an $s\in F$ and $x,y\in\Omega$, we know that there are elements $a,b,c\in D$ with $\|a-x\|<\delta$, $\|c-y\|<\delta$ and $\|d-\alpha_{s}(x)\|<\delta$. Note that $\|\alpha_{s}(a)-d\|\leq \|\alpha_{s}(a)-\alpha_{s}(x)\|+\|\alpha_{s}(x)-d\|<2\delta$.  We may now verify the approximate equivariance:
\begin{align*}\|\sigma_{s}(\varphi(x))-\varphi(\alpha_{s}(x))\|&\leq \|\sigma_{s}(\varphi(x))-\sigma_{s}(\varphi(a))\|
+\|\sigma_{s}(\varphi(a))-\varphi(d)\|+\|\varphi(d)-\varphi(\alpha_{s}(x))\|\\&<2\delta +\|\sigma_{s}(\varphi(a))-\varphi(d)\|=
 2\delta +\|(u_{s,t}\alpha_{s}(a)u_{s,t}^{*})_{t\in\Lambda}-(d)_{t\in\Lambda}\|\\&\leq
 2\delta +\max_{t\in\Lambda}\|u_{s,t}\alpha_{s}(a)u_{s,t}^{*}-d\|\\
 &\leq 2\delta +\max_{t\in\Lambda}\big\{\|u_{s,t}\alpha_{s}(a)u_{s,t}^{*}-\alpha_{s}(a)\|+\|\alpha_{s}(a)-d\| \big\}\\
 &\leq 2\delta +2\|a\|\delta + 2\delta\leq 2\delta(2+C)<\varepsilon.
\end{align*}

As for approximate multiplicativity, a simple estimate gives $\|xy-ab\|<2C\delta$ as well as $\|\varphi(a)\varphi(b)-\varphi(x)\varphi(y)\|<\delta$. Also note that $\varphi$ is multiplicative on $D$, so
\begin{align*}\|\varphi(xy)-\varphi(x)\varphi(y)\|&\leq\|\varphi(xy)-\varphi(ab)\|+\|\varphi(a)\varphi(b)-\varphi(x)\varphi(y)\|\\
&\leq\|xy-ab\|+\|\varphi(a)\varphi(b)-\varphi(x)\varphi(y)\|<4C\delta<\varepsilon.
\end{align*}

Finally, since $\|\varphi(a)\|=\|a\|$, we have
\[\big|\|\varphi(x)\|-\|x\|\big|\leq\big|\|\varphi(x)\|-\|\varphi(a)\|\big|+\big|\|a\|-\|x\|\big|\leq \|\varphi(x)-\varphi(a)\|+\|a-x\|\leq 2\delta<\varepsilon,\]
which concludes the proof.
\end{proof}

As mentioned above, we pay attention to residually finite actions, and quasidiagonal actions in the non-commutative case, for such actions will determine the structure of the resulting reduced crossed product algebras. To employ Theorem~\ref{KerrNowak}, we need a quasidiagonal \cstar-system $(A,\Gamma)$ where $C^*_{\lambda}(\Gamma)$ is MF. A remarkable result of Haagerup and Thorbj{\o}rnsen in~\cite{HT} states that the reduced group \cstar-algebra $\cstar_{\lambda}(\mathbb{F}_{r})$ is MF. However, by Rosenberg's result, $\cstar_{\lambda}(\mathbb{F}_{r})$ is not quasidiagonal when $r\geq2$. Therefore a reduced crossed product where the acting group is a non-abelian free group can never be quasidiagonal. A wonderful result connecting the Brown-Douglas-Fillmore theory of extensions \cite{BDF} to quasidiagonal \cstar-algebras and MF algebras reads as follows. A proof of this result can be found in~\cite{B2}.

\begin{thm} Let $B$ be a unital separable MF algebra which fails to be quasidiagonal. Then $\Ext(B)$ is not a group.
\end{thm}

We have the following corollaries.

\begin{cor}Let $A$ be a unital AF algebra satisfying $\overline{\Inn(A)}=\Aut(A)$, then $A\rtimes_{\lambda,\alpha}\mathbb{F}_{r}$ is an MF algebra. In particular, if $A$ is AF with $K_{0}(A)$ totally ordered and Archimedean, or if $A$ is UHF, then $A\rtimes_{\lambda,\alpha}\mathbb{F}_{r}$ is always MF. For such algebras $A$ and $r\geq2$ we have that $\Ext(A\rtimes_{\lambda,\alpha}\mathbb{F}_{r})$ is not a group.
\end{cor}

\begin{cor}Let $A$ be an AF-algebra, $r\in\mathbb{N}$ and $\alpha:\mathbb{F}_{r}\rightarrow\Aut(A)$ an action with the pseudo-orbit property. Then $A\rtimes_{\lambda,\alpha}\mathbb{F}_{r}$ is an MF algebra. If $r\geq2$ then $\Ext(A\rtimes_{\lambda,\alpha}\mathbb{F}_{r})$ is not a group.
\end{cor}

For actions on nuclear algebras, we aim to show that QD and MF actions coincide. Before embarking on the details, a bit of notation is in order. Given a (separable) MF algebra $B$, by an \emph{MF approximating sequence for $B$} we mean a sequence of $\ast$-linear, unital maps $(\psi_{n}:B\rightarrow \mathbb{M}_{k_{n}})_{n\geq1}$ which are asymptotically multiplicative and asymptotically isometric. If the $\psi_{n}$ are completely positive, then $B$ is a quasidiagonal algebra and the sequence $(\psi_{n})_{n\geq1}$ will be referred to as a \emph{QD approximating sequence}.

\begin{lem}\label{approx seq} Let $A$ be a unital, separable, nuclear MF algebra. Suppose $(\psi_{n}:A\rightarrow \mathbb{M}_{k_{n}})_{n\geq1}$ is an MF approximating sequence for $A$. Then there exists a QD approximating sequence $(\varphi_{n}:A\rightarrow \mathbb{M}_{k_{n}})_{n\geq1}$ for $A$ satisfying
\[\|\varphi_{n}(a)-\psi_{n}(a)\|\longrightarrow 0 \qquad \forall a\in A.\]
\end{lem}

\begin{proof} If $\pi:\prod_{n=1}^{\infty}\mathbb{M}_{k_{n}}\rightarrow\frac{\prod_{n=1}^{\infty}\mathbb{M}_{k_{n}}}{\bigoplus_{n=1}^{\infty}\mathbb{M}_{k_{n}}}$
 denotes the canonical quotient mapping, the MF approximating sequence $(\psi_{n})_{n\geq1}$ provides an embedding \[\Psi:A\hookrightarrow\frac{\prod_{n=1}^{\infty}\mathbb{M}_{k_{n}}}{\bigoplus_{n=1}^{\infty}\mathbb{M}_{k_{n}}},\]
namely $\Psi(a):=\pi((\psi_{n}(a))_{n\geq1})$. Now nuclearity of $A$ ensures a u.c.p. lifting
\[\Phi:A\rightarrow \prod_{n=1}^{\infty}\mathbb{M}_{k_{n}}\]
with $\pi\circ\Phi=\Psi$. Set for each $n$, $\varphi_{n}:=\pi_{n}\circ\Phi:A\rightarrow\mathbb{M}_{k_{n}}$, where $\pi_{n}:\prod_{m=1}^{\infty}\mathbb{M}_{k_{m}}\rightarrow\mathbb{M}_{k_{n}}$ is the natural projection mapping. The maps $\varphi_{n}$ are clearly u.c.p, and note that for each $a$ in $A$,
\[\pi((\psi_{n}(a))_{n})=\Psi(a)=\pi\circ\Phi(a)=\pi((\varphi_{n}(a))_{n}),\]
which means that $(\varphi_{n}(a)-\psi_{n}(a))_{n\geq1}\in \bigoplus_{n=1}^{\infty}\mathbb{M}_{k_{n}}$ for each $a$, that is
\[\|\varphi_{n}(a)-\psi_{n}(a)\|\longrightarrow 0 \qquad \forall a\in A.\]
From this, the approximating properties of the sequence $(\varphi_{n})_{n\geq1}$ follow from those of $(\psi_{n})_{n\geq1}$. Indeed, for each $a,b\in A$
\begin{align*}\|\varphi_{n}(ab)-\varphi_{n}(a)\varphi_{n}(b)\|\leq \|\varphi_{n}(ab)&-\psi_{n}(ab)\|+\|\psi_{n}(ab)-\psi_{n}(a)\psi_{n}(b)\|\\&+\|\psi_{n}(a)\psi_{n}(b)-\varphi_{n}(a)\varphi_{n}(b)\|
\end{align*}
with each term tending to zero as $n\rightarrow\infty$. Note that one needs a standard $\varepsilon/3$ argument to show that the last term tends to zero. Also
\[\big|\|\varphi_{n}(a)\|-\|a\|\big|\leq\|\varphi_{n}(a)-\psi_{n}(a)\|+\big|\|\psi_{n}(a)\|-\|a\|\big|\stackrel{n\rightarrow\infty}{\longrightarrow}0\]
for every $a\in A$.
\end{proof}

\begin{prop}\label{QDactionequalsMFaction} Let $(A,\Gamma,\alpha)$ be a \cstar-dynamical system with $A$ nuclear and separable. Then $\alpha$ is a quasidiagonal action if and only if it is an MF action.
\end{prop}

\begin{proof} That QD implies MF is obvious. Assume that $\alpha$ is MF. We then have an MF approximating sequence $(\psi_{n}:A\rightarrow \mathbb{M}_{k_{n}})_{n\geq1}$ as well as a sequence of actions $\gamma_{n}\curvearrowright\mathbb{M}_{k_{n}}$ with
\[\|\gamma_{n,s}(\psi_{n}(a))-\psi_{n}(\alpha_{s}(a))\|\stackrel{n\rightarrow\infty}{\longrightarrow} 0\qquad \forall a\in A, \forall s\in\Gamma.\]

Use the above Lemma~\ref{approx seq} to generate a QD approximating sequence $(\phi_{n}:A\rightarrow \mathbb{M}_{k_{n}})_{n\geq1}$ with $\|\varphi_{n}(a)-\psi_{n}(a)\|\rightarrow 0$ for every $a\in A$. For a fixed $a\in A$ and $s\in\Gamma$, a simple estimate now gives
\begin{align*}\|\gamma_{n,s}(\phi_{n}(a))-\phi_{n}(\alpha_{s}(a))\|\leq
\|\gamma_{n,s}(\phi_{n}(a))&-\gamma_{n,s}(\psi_{n}(a))\|+\|\gamma_{n,s}(\psi_{n}(a))-\psi_{n}(\alpha_{s}(a))\|\\&+\|\psi_{n}(\alpha_{s}(a))-\phi_{n}(\alpha_{s}(a))\|
\end{align*}
which tends to zero as $n\rightarrow\infty$ since $\|\gamma_{n,s}(\phi_{n}(a))-\gamma_{n,s}(\psi_{n}(a))\|\leq\|\varphi_{n}(a)-\psi_{n}(a)\|$ which goes to zero. The action is thus QD.
\end{proof}

We mention one more example taken from~\cite{HadS} which stems from~\cite{PV}.

\begin{example} Consider an action of the integers $\alpha:\mathbb{Z}\rightarrow\Aut(A)$ which admits an \emph{almost periodic} condition: there is a natural sequence $(n_{k})_{k\geq1}$ for which $(\alpha_{n_{k}})_{k}\rightarrow\id_{A}$ in $\Aut(A)$ as $k\rightarrow\infty$. Call such an action (AP). Pimsner and Voiculescu showed (see~\cite{PV}) that if $A$ is separable, unital, and quasidiagonal and $\mathbb{Z}\curvearrowright A$ satisfies (AP), then $A\rtimes_{\lambda}\mathbb{Z}$ is also quasidiagonal. Hadwin and J. Shen proved an analogous result in the context of MF algebras. In Theorem 4.2 of~\cite{HadS}, they prove that if $A$ is MF, unital and finitely generated $\mathbb{Z}\curvearrowright A$ satisfies condition (AP), then $A\rtimes_{\lambda}\mathbb{Z}$ is again MF. From their work and applying Proposition~\ref{MF cross implies MF act}, we conclude that (AP) actions of the integers on unital QD algebras are QD, and (AP) actions of the integers on unital finitely generated MF algebras are MF. We mention that this notion of an almost periodic action was generalized to actions of amenable countable residually finite discrete groups by Orfanos \cite{Or} where he extended Pimsner and Voiculescu's result.
\end{example}

The results obtained thus far have a concise formulation when the underlying algebra is nuclear.

\begin{thm}\label{Summarypart1} Let $A$ be a unital separable nuclear \cstar-algebra, $\Gamma$ a countable discrete group and $\alpha:\Gamma\rightarrow\Aut(A)$  an action. The following hold.
\begin{enumerate}
\item $A\rtimes_{\lambda,\alpha}\Gamma$ is MF if and only if $\cstar_{\lambda}(\Gamma)$ is MF and $\alpha$ is MF.
\item $A\rtimes_{\lambda,\alpha}\Gamma$ is QD if and only if $\cstar_{\lambda}(\Gamma)$ is QD and $\alpha$ is QD.
\item $A\rtimes_{\lambda,\alpha}\Gamma$ is RFD if and only if $\cstar_{\lambda}(\Gamma)$ is RFD and $\alpha$ is RFD.
\end{enumerate}
When $A=C(X)$ is abelian, $\alpha$ is RFD if and only if the induced action $\Gamma\curvearrowright X$ is RFD. Moreover, if $\Gamma=\mathbb{F}_{r}$ and $X$ is a zero-dimensional metrizable space, then $\alpha$ is QD if and only if the induced action $\Gamma\curvearrowright X$ is residually finite.
\end{thm}

\begin{proof} (1): This follows from Theorem~\ref{KerrNowak}, Proposition~\ref{QDactionequalsMFaction} and Proposition~\ref{MF cross implies MF act}. Recall that being MF passes to subalgebras.

(2): If $\alpha$  is QD and $\cstar_{\lambda}(\Gamma)$ is QD then $\Gamma$ is amenable by Rosenberg's result and $A\rtimes_{\lambda,\alpha}\Gamma$ is MF by Theorem~\ref{KerrNowak}. Since $A$ is nuclear and $\Gamma$ is amenable, then $A\rtimes_{\lambda,\alpha}\Gamma$ is nuclear. Now recall that nuclear and MF implies QD. The converse is again Proposition~\ref{MF cross implies MF act}.

(3): This is Proposition~\ref{KerrNowak}.

A residually finite action $\Gamma\curvearrowright X$ by any discrete group on any compact metric space always induces a quasidiagonal action on $C(X)$ as shown in Proposition~\ref{RFactionimpliesQDaction}. Moreover, it is shown in~\cite{KN} that if the reduced crossed product $C(X)\rtimes_{\lambda,\alpha}\mathbb{F}_{r}$ is MF, the induced action $\mathbb{F}_{r}\curvearrowright X$ is residually finite, provided that $X$ is a zero-dimensional compact space.
\end{proof}

\section{K-Theoretical Dynamics}

In this section our aim is to model classical and noncommutative \cstar-dynamics $K$-theoretically. In the presence of sufficiently many projections, the properties of residually finite, RFD, and MF actions admit simple $K$-theoretic characterizations that will aid us to prove structure theorems for the resulting reduced crossed products. Proposition~\ref{K-RFD} below shows how RFD systems $(A,\Gamma)$ admit $\Gamma$-invariant, integer-valued states on $K_{0}(A)$. Quasidiagonal actions are likewise described via local, $\Gamma$-invariant faithful states. For an action of a free group on a AF algebra this characterization leads to the coboundary condition $H_{\sigma}\cap K_{0}(A)^{+}=\{0\}$ from which the main result Theorem~\ref{Main2} ensues.

\subsection{Commutative Case}

We first restrict our attention to transformation groups $h:\Gamma\curvearrowright X$ where $X$ is zero-dimensional. As usual we shall denote by $\alpha$ the corresponding action on $A=C(X)$ and by $\hat\alpha$ the induced action on $K_{0}(A)$. Introducing some further notation for this result, if $A$ is any \cstar-algebra, write
\[\Sigma(A)=\{[p]_{0}: p \in \mathcal{P}(A)\}\]
for the scale of $A$, and given subsets $F\subset\Gamma$, and $S\subset K_{0}(A)$ write
\[S_{F}=\{\hat\alpha_{t}(x):\ x\in S,\  t\in F\cup\{e\}\ \}\]
for the subset of $K_{0}(A)$ containing $S$ and all $F$-iterates of $S$.

\begin{prop} Let $X$ be a zero-dimensional compact metrizable space, and $\Gamma$ a discrete group. Suppose $h:\Gamma\curvearrowright X$ is a continuous action with induced action $\alpha:\Gamma \rightarrow\Aut(C(X))$. Then the following statements are equivalent.

\begin{enumerate}
\item $\Gamma\curvearrowright X$ is residually finite.

\item Given finite subsets $S\subset K_{0}(C(X))^{+}$ and $F \subset \Gamma$ there exist $d$ in $\mathbb{N}$, an action $\sigma:\Gamma \rightarrow\OAut(\mathbb{Z}^{d})$ of ordered abelian groups, and a morphism of ordered abelian groups $\beta:K_{0}(A)\rightarrow \mathbb{Z}^{d}$ such that
\begin{enumerate}
\item $\beta\circ\hat\alpha_{t}(g)=\sigma_{t}\circ\beta(g)$ for each $g \in S$ and $t\in F$,
\item  $\beta(g)\neq 0$ for every $0\neq g\in S$.
\end{enumerate}

\item Given finite subsets $S \subset K_{0}(C(X))^{+}$ and $F \subset \Gamma$ there exist $d$ in $\mathbb{N}$, a subgroup $H\leq K_{0}(C(X))$, along with an action $\sigma:\Gamma \rightarrow\OAut(\mathbb{Z}^{d})$ by ordered abelian group automorphisms, and a positive group homomorphism $\beta:H\rightarrow \mathbb{Z}^{d}$ such that
\begin{enumerate}
\item $[1]\in H$, $S_{F}\subset H$, and $\beta([1])=(1,1,\dots,1)$,
\item $\beta\circ \hat\alpha_{t}(g)=\sigma_{t}\circ\beta(g)$ for each $g \in S$ and $t\in F$,
\item $\beta(g)\neq 0$ for every $0\neq g\in S$.
\end{enumerate}
\end{enumerate}

\end{prop}

\begin{proof} $(1)\Rightarrow(2)$: First consider $S'=\{[p_{j}]:p_{j}\in \mathcal{P}(A), j=1,\dots n\}\subset \Sigma(C(X))$, a finite subset of the scale of $C(X)$, and $F\subset\Gamma$ a finite subset. Let $0<\varepsilon<1$. Since $h$ is residually finite, by the proof of Proposition~\ref{RFactionimpliesQDaction} in~\cite{KN} there is a unital $\ast$-homomorphism $\varphi:A\rightarrow \mathbb{C}^{d}$ for some $d\in\mathbb{N}$, and an action $\gamma:\Gamma\rightarrow\Aut(\mathbb{C}^{d})$ such that for each $j\in\{1,\dots,n\}$ and $t\in F$
\begin{gather*}\|\varphi(p_{j})\| > \|p_{j}\|-\varepsilon,\\
\|\varphi(\alpha_{t}(p_{j}))-\gamma_{t}(\varphi(p_{j}))\|<\varepsilon.
\end{gather*}
Applying the $K_{0}$ functor yields a positive group homomorphism $\beta:=K_{0}(\varphi):K_{0}(A)\rightarrow K_{0}(\mathbb{C}^{d})\cong \mathbb{Z}^{d}$, with $\beta([1_{A}])=[\varphi(1_{A})]=[1_{\mathbb{C}^{d}}]\cong (1,\dots,1)$. As in the above discussion we also have an induced action $K_{0}(\gamma):\Gamma\rightarrow\OAut(K_{0}(\mathbb{C}^{d}))$. Write $\sigma_{t}=K_{0}(\gamma)(t)=K_{0}(\gamma_{t})$. After composing by a suitable isomorphism of ordered abelian groups, we may assume $\beta$ takes values in $\mathbb{Z}^{d}$, and $\sigma_{t}\in \OAut(\mathbb{Z}^{d})$. We may now verify equivariance: for $t\in F$ and every $j$ we have
\begin{align*}
  \beta\circ \hat\alpha_{t}([p_{j}])&=\hat\varphi\circ \hat\alpha_{t}([p_{j}])=[\varphi\circ\alpha_{t}(p_{j})]\\&=[\gamma_{t}\circ\varphi(p_{j})]
  =\hat\gamma_{t}\circ \hat\varphi([p_{j}])=\sigma_{t}\circ\beta([p_{j}]).
\end{align*}

Suppose $\beta([p_{j}])=0$ for some $j$. Then by definition of $\beta$, $[\varphi(p_{j})]=0$ in $K_{0}(\mathbb{C}^{d})$, which gives $\varphi(p_{j})\sim_{0}0$ and so $\varphi(p_{j})=0$. However, we read above that $\|\varphi(p_{j})\|>\|p_{j}\|-\varepsilon$, which is absurd when $p_{j}\neq 0$.

Since $C(X)$ is AF, the positive cone $K_{0}(C(X))^{+}$ is generated by its scale. Therefore, if $S=\{[q_{i}]\}_{i=1}^{m}\subset K_{0}(C(X))^{+}$, for each $i$ there are elements of the scale $\{[p_{ij}]\}_{j=1}^{n_{i}}$ and positive integers $k_{ij}$ with $[q_{i}]=\sum_{j=1}^{n_{i}}k_{ij}[p_{ij}]$. Set $S'=\{[p_{ij}] : i=1,\dots,m,\quad j=1,\dots, n_{i}\}$ and find $d$, $\beta$, and $\sigma$ as above. Clearly $\beta$ remains equivariant on $S$. Since $\beta$ is faithful on $S'$, it remains faithful on $S$.\newline

$(2)\Rightarrow(3)$: This is obvious; simply take $H=K_{0}(C(X))$.\newline

$(3)\Rightarrow(1)$: Fix a finite set $F\subset\Gamma$ and let $\varepsilon>0$. By compactness and the zero-dimensionality of $X$, we can choose a clopen partition $X=\bigsqcup_{j=1}^{n}Y_{j}$ with $\diam(Y_{j})<\varepsilon/2$. Set $p_{j}=\textbf{1}_{Y_{j}}$ and note that these are orthogonal projections with $\sum_{j}^{n}p_{j}=\textbf{1}_{X}$. Consider now
\[B=\cstar(\{\alpha_{s}(p_{j}):\;s\in F, j=1,\dots,n\})\quad\mbox{and}\quad S=\{[p_{1}],\dots,[p_{n}]\}\subset K_{0}(C(X))^{+}.\]
Apply $(3)$ and obtain  suitable $d$, $H$, $\beta$, and $\sigma$. If $\iota:B\hookrightarrow C(X)$ denotes inclusion, $\hat\iota=K_{0}(\iota):K_{0}(B)\rightarrow K_{0}(C(X))$ is a positive group homomorphism. By hypothesis, the subgroup $H\leq K_{0}(C(X))$ contains all the classes of iterates $\{[\alpha_{s}(p_{j})]: s\in F\cup\{e\}, j=1,\dots, n \}$ and $[1_{A}]$. This guarantees $H$ will contain the image of $\hat\iota$, and so we can therefore compose and define the positive group homomorphism
\[\tau:=\beta\circ\hat\iota:K_{0}(B)\rightarrow \mathbb{Z}^{d}.\]
After composing with a suitable isomorphism of ordered abelian groups, we may assume
\[\tau:K_{0}(B)\rightarrow K_{0}(\mathbb{C}^{d}),\qquad \sigma:\Gamma \rightarrow \OAut(K_{0}(\mathbb{C}^{d})),\]
and these satisfy $\tau([1_{A}])=[1_{\mathbb{C}^{d}}]$, and $\sigma_{t}([1_{A}])=[1_{\mathbb{C}^{d}}]$ for each $t$ in $\Gamma$. By lemma 1.3.4 of~\cite{Ro} there is a unital $\ast$-morphism $\varphi: B\rightarrow \mathbb{C}^{d}$ with $K_{0}(\varphi)=\tau$, and an action $\gamma:\Gamma\rightarrow\Aut(\mathbb{C}^{d})$ with $K_{0}(\gamma_{t})=\sigma_{t}$. We then extend $\varphi$ to all of $C(X)$. The conditions then read as follows: for each $t\in F$ and $j=1,\dots,n$
\begin{align*}
[\varphi(\alpha_{t}(p_{j}))]&=\hat\varphi([\alpha_{t}(p_{j})])=\tau([\alpha_{t}(p_{j})])=\beta\circ \hat\iota([\alpha_{t}(p_{j})])\\&=\beta\circ \hat\alpha_{t}([p_{j}])=\sigma_{t}\circ\beta([p_{j}])=\sigma_{t}\circ\beta\circ \hat\iota([p_{j}])\\
&=\sigma_{t}\circ\tau([p_{j}])=\hat\gamma_{t}\circ \hat\varphi([p_{j}])=[\gamma_{t}(\varphi(p_{j}))].
\end{align*}
This equality holds true in $K_{0}(\mathbb{C}^{d})$ where there is cancellation, whence $\varphi(\alpha_{t}(p_{j}))\sim_{0}\gamma_{t}(\varphi(p_{j}))$, and commutativity then yields the equality $\varphi(\alpha_{t}(p_{j}))=\gamma_{t}(\varphi(p_{j}))$. Moreover, if  $\varphi(p_{j})=0$, it follows that
\[\beta([p_{j}])=\beta\circ \hat\iota([p_{j}])=\tau([p_{j}])=\hat\varphi([p_{j}])=[\varphi(p_{j})]=0,\]
which entails, by the condition on $\beta$, that $[p_{j}]=0$ and thus $p_{j}=0$. Thus $\|\varphi(p_{j})\|=1$ whenever $p_{j}$ is a non-zero projection.

Let $\zeta:\{1,\dots,d\}\rightarrow X$ be the map for which $\varphi(f)=f\circ\zeta$. Moreover, there is an action $\Gamma\curvearrowright\{1,\dots,d\}$ such that $\gamma_{t}(g)(z)=g(t.^{-1}z)$ for every $z\in\{1,\dots,d\}$ and $t\in\Gamma$. The above equivariance of $\varphi$ implies that for each $j=1,\dots,n$, $t\in F$ and $z\in\{1,\dots,d\}$

\begin{align*}p_{j}(\zeta(t^{-1}.z))=\varphi(p_{j})(t^{-1}.z)=\gamma_{t}(\varphi(p_{j}))(z)
&=\varphi(\alpha_{t}(p_{j}))(z)\\&=\alpha_{t}(p_{j})(\zeta(z))=p_{j}(t^{-1}.\zeta(z)).
\end{align*}
This shows that for such $t$ and $z$, $d(t.^{-1}\zeta(z),\zeta(t.^{-1}z))<\varepsilon$, for otherwise $t^{-1}.\zeta(z)$ and $\zeta(t^{-1}.z)$ would be separated by some $p_{j}$ and the above equality would fail. Next, the faithfulness of $\varphi$ means that for each fixed $j$
\[\max_{z\in\{1,\dots,d\}}|p_{j}(\zeta(z))|=\|\varphi(p_{j})\|=1.\]
This proves that $X\subset_{\varepsilon}\zeta(\{1,\dots,d\})$, for if $x\in X$, $x$ belongs to some $Y_{j_{0}}$ and the above equality applied to $p_{j_{0}}$ ensures that there is a $z_{0}$ with $\zeta(z_{0})\in Y_{j_{0}}$ which gives $d(\zeta(z_{0}),x)<\varepsilon$. The action is thus residually finite, completing the proof.

\end{proof}

\subsection{Perturbation Lemmas}
Modeling non-commutative \cstar-dynamics at the K-theoretical level will involve some perturbation results. Recall that two projections determine the same class in $K_{0}$ provided that they are sufficiently close. This allows us some much needed flexibility when applying the $K_{0}$ functor. The next few results are sufficient for our purposes. The first perturbation lemma is quite standard, and may be found in Davidson's book~\cite{Da}, Lemma III.3.2. We therefore state it without proof.

\begin{lem}\label{mtx units pert}Given $\varepsilon>0$ and $n\in\mathbb{N}$, there exists a $\delta=\delta(\varepsilon,n)$ with the following property: Given a unital \cstar-algebra $A$ and \cstar-subalgebras $C, D \subset A$ with $\dim(C)=n$ and system of matrix units $\mathcal{E}$ for $C$ satisfying $\mathcal{E}\subset_{\delta}D$, there is a unitary $u\in A$ with $\|1-u\|<\varepsilon$ such that $uCu^{*}\subset D$.
\end{lem}

This next result is crucial to the main theorem of this paper, and so we offer a full detailed proof.

\begin{lem}\label{approx ucp by homo}Let $B\cong \mathbb{M}_{n_{1}}\oplus\cdots\oplus\mathbb{M}_{n_{s}}$ be a finite dimensional \cstar-algebra with $\dim(B)=d$ and system of canonical matrix units $\mathcal{E}=\{e_{i,j}^{r}\}$. Then for every $\varepsilon >0$, there is a $\delta=\delta(\varepsilon,d) >0$ with the following property:

Given a u.c.p.\ map $\varphi:B\rightarrow \mathbb{M}_{k}$ which is approximately multiplicative on $\mathcal{E}$ within $\delta$, there is a unital $\ast$-homomorphism $\sigma:B\rightarrow \mathbb{M}_{k}$ with

 \[\|\sigma(x)-\varphi(x)\|<\varepsilon\|x\|\qquad \mbox{for every}\; x\in B.\]
\end{lem}

\begin{proof} Let $B$ and $\varepsilon>0$ be given, we will later choose an appropriate $\delta$ depending only on $\varepsilon$ and on $d=$dim$(B)$. By Stinespring's dilation theorem, the u.c.p. map $\varphi:B\rightarrow \mathbb{M}_{k}$ is the compression of a unital $\ast$-homomorphism. More precisely, there is an isometry $V:\ell_{2}^{k}\rightarrow \ell_{2}^{l}$ and a unital $\ast$-homomorphism $\pi:B\rightarrow \mathbb{M}_{l}$ such that $\varphi(x)=V^{*}\pi(x)V$ for every $x\in B$. Let $P=VV^{*}$ denote the Stinespring projection in $\mathbb{M}_{l}$.\\

\noindent \textbf{Claim 1.} $\|[P,\pi(e)]\|<\sqrt\delta$ for every $e\in\mathcal{E}$, provided $\varphi$ is approximately multiplicative on $\mathcal{E}$ within $\delta$.

Using the identity $Pa-aP=Pa(1-P)-(1-P)aP$ for $a\in\mathbb{M}_{l}$, we get
\[\|Pa-aP\|=\max\{\|Paa^{*}P-PaPa^{*}P\|^{1/2},\|Pa^{*}aP-Pa^{*}PaP\|^{1/2}\}\]
If $e\in\mathcal{E}$, so is $e^{*}$, so setting $a=\pi(e)$, we get
\begin{align*}
\|P\pi(e)\pi(e)^{*}P-P\pi(e)P\pi(e)^{*}P\|&=\|VV^{*}\pi(ee^{*})VV^{*}-VV^{*}\pi(e)VV^{*}\pi(e^{*})VV^{*}\|\\
&=\|V\varphi(ee^{*})V^{*}-V\varphi(e)\varphi(e^{*})V^{*}\|\\&=\|V(\varphi(ee^{*})-\varphi(e)\varphi(e^{*}))V^{*}\|\\
&\leq \|\varphi(ee^{*})-\varphi(e)\varphi(e^{*})\|<\delta.
\end{align*}
Similarly, $\|P\pi(e)^{*}\pi(e)P-P\pi(e)^{*}P\pi(e)P\|<\delta$, and together with the above estimate we get $\|P\pi(e)-\pi(e)P\|<\sqrt\delta$ as claimed.\newline

\noindent \textbf{Claim 2.} Let $C=\pi(B)\subset \mathbb{M}_{l}$, then dim$(C)\leq d$ and $\|[P,u]\|<\sqrt\delta d$ for every $u\in \Ball(C)$, in particular for every unitary $u\in \mathcal{U}(C)$. \newline

If $u\in \Ball(C)$, we can lift $u$ to an $x\in\Ball(B)$ with $\pi(x)=u$. Write
\[x=\sum_{i,j,r}\alpha_{i,j}^{(r)}e_{i,j}^{(r)}, \qquad |\alpha_{i,j}^{(r)}|\leq 1.\]
Straightforward estimates yield
\begin{align*}
\|Pu-uP\|&=\|P\pi(x)-\pi(x)P\|=\|\sum_{i,j,r}\alpha_{i,j}^{(r)}(P\pi(e_{i,j}^{(r)})-\pi(e_{i,j}^{(r)})P)\|\\
&\leq\sum_{i,j,r}|\alpha_{i,j}^{(r)}|\|P\pi(e_{i,j}^{(r)})-\pi(e_{i,j}^{(r)})P\|\leq d\sqrt\delta
\end{align*}
where we've used Claim 1 and the fact that $|\alpha_{i,j}^{(r)}|\leq 1$. This proves Claim 2.

Now $C\subset\mathbb{M}_{l}$ is a finite dimensional subalgebra, so we have a conditional expectation $\mathbb{E}:\mathbb{M}_{l}\rightarrow C^{'}\cap\mathbb{M}_{l}$ given by
\[\mathbb{E}(a) = \int_{\mathcal{U}(C)}uau^{*}du.\]
where $du$ is the normalized Haar measure on $\mathcal{U}(C)$. Using the estimate from Claim 2 we have
\begin{align*}
\|\mathbb{E}(P)-P\|&=\bigg\|\int_{\mathcal{U}(C)}uPu^{*}du - \int_{\mathcal{U}(C)}Pdu\bigg\|=\bigg\|\int_{\mathcal{U}(C)}(uPu^{*}-P)du\bigg\|\\&\leq \int_{\mathcal{U}(C)}\|uPu^{*}-P\|du=\int_{\mathcal{U}(C)}\|uP-Pu\|du \leq d\sqrt\delta.
\end{align*}

Now let $0<\eta=\eta(\varepsilon)<1$, to be determined later. We know from standard perturbation results that there is a $\delta^{'}>0$ with the following property: if $A$ is any unital \cstar-algebra, $B\subset A$ is a unital subalgebra and $p\in \mathcal{P}(A)$ a projection with $\|p-b\|<\delta^{'}$, then there is a projection $q\in \mathcal{P}(B)$ with $\|p-q\|<\eta$. Making sure that $d\sqrt\delta<\delta^{'}$, there is a projection $q\in \mathcal{P}(C^{'}\cap\mathbb{M}_{l})$ with $\|P-q\|<\eta$. We may then find a unitary $u$ in $\mathbb{M}_{l}$ with $u^{*}Pu=q$ and $\|1-u\|\leq \sqrt 2 \eta$. Now we define
\[\sigma:B\rightarrow \mathbb{M}_{k}\qquad \sigma(b)=V^{*}u\pi(b)u^{*}V.\]
We claim that $\sigma$ is a unital $\ast$-homomorphism. Indeed, $\sigma(1)=V^{*}u\pi(1)u^{*}V=V^{*}uu^{*}V=V^{*}V=1$, and for $a$ and $b$ in $B$,
\begin{align*}
\sigma(a)\sigma(b)&=V^{*}u\pi(a)u^{*}VV^{*}u\pi(b)u^{*}V=V^{*}u\pi(a)u^{*}Pu\pi(b)u^{*}V\\
&=V^{*}u\pi(a)q\pi(b)u^{*}V=V^{*}u\pi(a)\pi(b)qu^{*}V=V^{*}u\pi(ab)qu^{*}V\\
&=V^{*}u\pi(ab)u^{*}PV=V^{*}u\pi(ab)u^{*}VV^{*}V=V^{*}u\pi(ab)u^{*}V=\sigma(ab).
\end{align*}
We now compute the desired perturbation
\begin{align*}
\|\sigma(x)-\varphi(x)\|&=\|V^{*}u\pi(x)u^{*}V-V^{*}\pi(x)V\|=\|V^{*}(u\pi(x)u^{*}-\pi(x))V\|\\
&\leq\|u\pi(x)u^{*}-\pi(x)\|\leq 2\|u-1\|\|\pi(x)\|\leq 2\sqrt2\eta\|x\|.
\end{align*}
Now simply choose $\eta$ so small that $2\sqrt2\eta<\varepsilon$.

\end{proof}

The next lemma is a straightforward application of spectral theory.

\begin{lem}\label{spectral lemma} Let $A$ be an AF algebra, $\mathcal{F}\subset A_{sa}$ a finite subset, and $\varepsilon>0$ be given. Then there is a finite dimensional subalgebra $B\subset A$ such that for every $a\in\mathcal{F}$ there are orthogonal projections $p_{1},\dots,p_{n}$ in $B$ and scalars $\lambda_{1},\dots,\lambda_{n}$ with

\[\|a-\sum_{j=1}^{n}\lambda_{j}p_{j}\|<\varepsilon.\]
\end{lem}

\subsection{Non-commutative Case}

We now wish to explore the $K$-theoretic expressions that describe  RFD, QD and MF \cstar-systems which in turn shed light on the structure of reduced crossed product. We begin with the more restrictive case; RFD actions.

\begin{defn} Let $A$ be a unital stably finite algebra. An action $\alpha:\Gamma\rightarrow\Aut(A)$ is said to be \emph{$K_{0}$-RFD} if the following holds: Given any non-zero $g\in K_{0}(A)^{+}$, there is a positive group homomorphism $\mu: K_{0}(A)\rightarrow\mathbb{Z}$ with
\begin{enumerate}
\item $\mu([1_{A}])>0$, and $\mu(g)>0$.
\item $\mu(\hat\alpha_{s}(x))=\mu(x)$ for every $x\in K_{0}(A)$.
\end{enumerate}
\end{defn}

\begin{prop}\label{K-RFD} Let $A$ be a unital stably finite algebra and $\alpha:\Gamma\rightarrow\Aut(A)$ an action. Consider the following properties:
\begin{enumerate}
\item $A\rtimes_{\lambda,\alpha}\Gamma$ is RFD.
\item The action $\alpha$ is RFD.
\item The action $\alpha$ is $K_{0}$-RFD.
\end{enumerate}

Then $(1)\Leftrightarrow(2)\Rightarrow(3)$. Moreover, if $A$ is AF and $\Gamma=\mathbb{F}_{r}$, then $(3)\Rightarrow(2)$, whence all three properties are equivalent.
\end{prop}

\begin{proof} $(1)\Leftrightarrow(2)$ was shown in Theorem~\ref{KerrNowak}.

$(2)\Rightarrow(3)$: Let $\alpha$ be an RFD action and let $g=[p]$ be a non-zero element in $K_{0}(A)$, in which case $p\neq0$. By amplifying the action we may assume that $p$ is a non-zero projection in $A$. Setting $\varepsilon=1/2$, there is a $\ast$-homomorphism $\pi:A\rightarrow\mathbb{M}_{d}$ and an inner action $\gamma\curvearrowright\mathbb{M}_{d}$ such that
\begin{enumerate}
\item[(i)] $\|\pi(q)\|\geq \|q\|-1/2=1/2,$
\item[(ii)] $\pi(\alpha_{s}(a))=\gamma_{s}(\pi(a))$ for every $s\in\Gamma$ and $a\in A$.
\end{enumerate}
Applying the $K_{0}$ functor we get a positive group homomorphism $\hat\pi:K_{0}(A)\rightarrow K_{0}(\mathbb{M}_{d})$ with $\hat\pi([1_{A}])=[1_{\mathbb{M}_{d}}]$. The action $\gamma$ induces the trivial action at the $K_{0}$-level so that condition (ii) implies $\hat\pi(\hat\alpha_{s}([q]))=\hat\pi([q])$ for every $q\in\mathcal{P}_{\infty}(A)$. Recall that $K_{0}(A)=K_{0}(A)^{+}-K_{0}(A)^{+}$, so that $\hat\pi(\hat\alpha_{s}(x))=\hat\pi(x)$ for every $x\in K_{0}(A)$. Now let
\[\mu=\beta\circ\hat\pi: K_{0}(A)\rightarrow\mathbb{Z}\]
where $\beta:(K_{0}(A),K_{0}(A)^{+},[1_{A}])\rightarrow(\mathbb{Z},\mathbb{Z}^+,d)$ is an isomorphism of ordered abelian groups. Clearly $\mu$ is a positive group homomorphism that satisfies the the required equivariance condition as well as $\mu([1])=d>0$ and $\mu(g)\geq0$. Also, by stable finiteness
\[\mu(g)=0 \Rightarrow \beta([\pi(q)])=0\Rightarrow [\pi(q)]=0\Rightarrow \pi(q)=0,\]
a contradiction.

We establish the implication $(3)\Rightarrow(2)$ in the case where $A$ is an AF algebra and $\Gamma=\mathbb{F}_{r}$ is a free group. Suppose now that $\alpha$ satisfies the $K_{0}$-RFD condition. Let $\varepsilon>0$ and $b=b^*\in A$. Since $A$ has real rank zero, there are orthogonal non-zero projections $p_{1},\dots,p_{n}$ in $A$ and scalars $t_{1},\dots,t_{n}$ such that
\[\|b-\sum_{j=1}^{n}t_{j}p_{j}\|<\varepsilon/2.\]
We may as well assume that $\|\sum_{j=1}^{n}t_{j}p_{j}\|=\max_{1\leq j\leq n}|t_{j}|=|t_{1}|$. Set $g=[p_{1}]$ and apply the $K_{0}$-RFD condition. We obtain a positive group homomorphism $\mu$ and set $\mu([1_{A}])=d$. By composing with an ordered group isomorphism we may suppose that $\mu$ takes values in $K_{0}(\mathbb{M}_{d})$ and $\mu([1])=[1_{\mathbb{M}_{d}}]$. By lemma 1.3.4 of~\cite{Ro} there is a unital $\ast$-homomorphism $\pi:A\rightarrow \mathbb{M}_{d}$ such that $\hat\pi=\mu$. Fix a generator $s_{j}\in\mathbb{F}_{r}$ where $1\leq j\leq r$. The two $\ast$-homomorphisms $\pi$ and $\pi\circ\alpha_{s_{j}}$ from $A$ to $\mathbb{M}_{d}$ agree at the $K$-theoretic level by condition (ii) of $K_{0}$-RFD. Utilizing once more lemma 1.3.4 of~\cite{Ro} there is a sequence of unitaries $(u_{n})_{n\geq1}$ in $\mathcal{U}(\mathbb{M}_{d})$ with
\[\Ad_{u_{n}}\circ\ \pi(a)\stackrel{n\rightarrow\infty}{\longrightarrow}\pi\circ\alpha_{s_{j}}(a)\quad\forall a\in A.\]
By the compactness of $\mathcal{U}(d)$, we may assume $\|u_{n}-u_{j}\|\rightarrow 0$ for some unitary $u_{j}\in\mathbb{M}_{d}$. Thus $\pi\circ\alpha_{s_{j}}(a)=\Ad_{u_{j}}\circ\ \pi(a)$ for every $a$ in $A$. By the universal property of the free group we may now define an inner action $\gamma:\mathbb{F}_{r}\curvearrowright\mathbb{M}_{d}$ by $\gamma_{s_{j}}=\Ad_{u_{j}}$. Thus $\gamma_{s}(\pi(a))=\pi(\alpha_{s}(a))$ holds for every $s\in\mathbb{F}_{r}$ and $a\in A$.

By condition (i) of $K_{0}$-RFD $\pi(p_{1})$ is a non-zero projection in $\mathbb{M}_{d}$. Write $c=\sum_{j=1}^{n}t_{j}p_{j}$, so $\|c\|=\|\pi(c)\|=|t_{1}|$. Then note that
\[\|b\|\leq\|b-c\|+\|c\|\leq\varepsilon/2+\|\pi(c)\|\leq \varepsilon/2+\|\pi(c)-\pi(b)\|+\|\pi(b)\|\leq \varepsilon+\|\pi(b)\|,\]
so that $\alpha$ is an RFD action.

\end{proof}

Recall that for a stably finite unital \cstar-algebra $A$, a \emph{state} on $(K_{0}(A), K_{0}(A)^{+},[1])$ is a group homomorphism $\beta:K_{0}(A)\rightarrow \mathbb{R}$ with $\beta(K_{0}(A)^{+})\subset\mathbb{R}^{+}$ and $\beta([1])=1$. Given an action $\Gamma\curvearrowright A$,  a state $\beta$ is $\Gamma$-\emph{invariant} if $\beta(\hat\alpha_{s}(x))=\beta(x)$ for every $x\in K_{0}(A)$ and $s\in\Gamma$. Therefore, in a sense, an RFD system $(A,\Gamma,\alpha)$ is one that admits an \emph{integer}-valued invariant state on $K$-theory that is locally faithful. A word of caution is in order. The fact that the invariant state emerging from an RFD action is integer valued is much more restrictive. We may consider minimal Cantor systems $(X,\mathbb{Z})$ for example. These always admit an invariant tracial state on $C(X)$, but the induced invariant state on $K_{0}(C(X))$ can never be integer valued by virtue of the previous proposition and the fact that $C(X)\rtimes\mathbb{Z}$ is simple.

We proceed to look at QD actions $K$-theoretically. As in the commutative case we focus our attention on AF algebras, in which case the notions of QD and MF actions coincide by Proposition~\ref{QDactionequalsMFaction}.

\begin{defn} Let $(A,\Gamma,\alpha)$ be a \cstar-dynamical system with $A$ unital. We say that $\alpha$ is $K_{0}$-QD if the induced action $\hat\alpha:\Gamma\rightarrow\OAut(K_{0}(A))$ satisfies the following condition:

Given finite subsets $S\subset K_{0}(A)^{+}$ and $F \subset \Gamma$ there is a subgroup $H\leq K_{0}(A)$, along with a group homomorphism $\beta:H\rightarrow \mathbb{Z}$ satisfying
\begin{enumerate}
\item $[1_{A}]\in H$ and $S_{F}\subset H$,
\item $\beta([1_{A}])>0$ and $\beta(g)>0$ for each $0\neq g\in S$,
\item $\beta(\hat\alpha_{s}(g))=\beta(g)$ for all $g\in S$ and $s\in F$.
\end{enumerate}
\end{defn}

One may thus paraphrase condition $K_{0}$-QD by saying that the action admits faithful $\Gamma$-invariant states in a local sense.

\begin{prop}\label{Main1} Let $A$ be a unital AF algebra, $\Gamma$ a discrete group and $\alpha:\Gamma\rightarrow\Aut(A)$ an MF-action. Then the induced action $\alpha$ is $K_{0}$-QD.
\end{prop}

\begin{proof} Approximately finite dimensional algebras are nuclear, so by Proposition~\ref{QDactionequalsMFaction} we may assume that $\alpha:\Gamma\rightarrow\Aut(A)$ is a quasidiagonal action. Fix a finite subset $F=\{e=s_{1},\dots,s_{m}\}\subset \Gamma$. We shall first consider a finite subset $S=\{[p_{1}],\dots,[p_{n}]\}\subset\Sigma(A)$ of the scale of $A$. Since $A$ is AF we can locate a unital finite dimensional subalgebra $B\subset A$ with $\{\alpha_{s_{i}}(p_{j})\}_{i,j}\subset_{1/4}B$. By perturbing, there are projections $q_{i,j}\in B$ with $\|\alpha_{s_{i}}(p_{j})-q_{i,j}\|<1/4$, whence $[\alpha_{s_{i}}(p_{j})]=[q_{i,j}]$ in $K_{0}(A)$. Consider the natural inclusion $\iota:B\hookrightarrow A$ which induces a positive group homomorphism
\[\hat\iota:K_{0}(B)\rightarrow K_{0}(A),\quad\mbox{where}\quad \hat\iota([q])=[\iota(q)]=[q].\]
Set $K=\ker(\hat\iota)\leq K_{0}(B)$. Since $K_{0}(B)$ is a finitely generated abelian group, so is $K$, say $K=\langle t_{1},\dots,t_{l}\rangle=\mathbb{Z}t_{1}+\cdots+\mathbb{Z}t_{l}$. By the continuity of the functor $K_{0}$, there is a finite dimensional subalgebra $D$ of $A$ containing $B$ with the following property: if $j:B\hookrightarrow D$ denotes inclusion, $\hat j(t_{i})=0$ in $K_{0}(D)$ for $i=1,\dots,l$.

Let $\varepsilon>0$ (to be determined later), and choose $\delta=\delta(\varepsilon, \dim(D))<\varepsilon$ according to the above perturbation Lemma~\ref{mtx units pert}. Also, set $G=\{e_{i,j}^{r}\}\cup\{q_{i,j}\}$ where $\{e_{i,j}^{r}\}$ is a system of matrix units for $D$. Since $\alpha$ is quasidiagonal, there are a positive integer $d$, a u.c.p.\ map $\varphi:A\rightarrow \mathbb{M}_{d}$, and an action $\gamma:\Gamma\rightarrow\Aut(\mathbb{M}_{d})$ such that for every $a,b\in G$ and $s\in F$
\begin{gather*}
\|\varphi(ab)-\varphi(a)\varphi(b)\|<\delta, \\
\|\varphi(a)\|>\|a\|-\delta,\\
\|\varphi(\alpha_{s}(a))-\gamma_{s}(\varphi(a))\|<\delta.
\end{gather*}

Utilizing the perturbation Lemma~\ref{approx ucp by homo}, there is a unital $\ast$-homomorphism $\pi:D\rightarrow \mathbb{M}_{d}$ such that
\[\|\pi(a)-\varphi(a)\|<\varepsilon \qquad \mbox{for every}\quad a\in \mbox{Ball}(D).\]
With the positive group homomorphism $\hat\pi:K_{0}(D)\rightarrow K_{0}(\mathbb{M}_{d})$ at hand, we define the subgroup $H=\hat\iota(K_{0}(B))\leq K_{0}(A)$ and the map
\[\beta:H\rightarrow K_{0}(\mathbb{M}_{d})\cong\mathbb{Z}\qquad \beta(\hat\iota(g)):=\hat\pi(\hat j(g)).\]

\noindent \textbf{Claim 1.} $\beta$ is a well defined group homomorphism satisfying condition $(2)$ of $K_{0}$-QD.

Given $g,g'\in K_{0}(B)$, with $\hat\iota(g)=\hat\iota(g')$, we have $0=\hat\iota(g)-\hat\iota(g')=\hat\iota(g-g')$, so that $g-g'\in K$. By construction, $\hat j(g-g')=0$, so $\hat j(g)=\hat j(g')$ and thus
\[\beta(\hat\iota(g))=\hat\pi(\hat j(g))=\hat\pi(\hat j(g'))=\beta(\hat\iota(g')),\]
showing that $\beta$ is well defined. Clearly $\beta$ is additive on $H$, and observe that
\[\beta([1_{A}])=\beta(\hat\iota([1_{A}]))=\hat\pi(\hat j([1_{A}]))=[\pi(1)]=[1_{\mathbb{M}_{d}}]\cong d.\]

Now let $0\neq g=[p_{j}]=[q_{1j}]$ be in $S$, which implies by cancellation that $q_{1j}$ is a non zero projection. Then $\beta(g)=\hat\pi(\hat j([q_{1j}]))=[\pi(q_{1j})]$ which is clearly positive in $K_{0}(\mathbb{M}_{d})$. Finally, if $[\pi(q_{1j})]=0$, then by cancellation, $\pi(q_{1j})=0$. However,
\[ |\|\pi(q_{1j})\|-\|q_{1j}\||\leq |\|\pi(q_{1j})\|-\|\varphi(q_{1j})\||+|\|\varphi(q_{1j})\|-\|q_{1j}\||<\epsilon+\delta<2\varepsilon,\]
and since $\|q_{1j}\|=1$, by choosing $\varepsilon <1/3$, $\pi(q_{1j})$ must be a non-zero projection as well, an absurdity. The Claim is thus proved.\\

We now verify the promised equivariance with the induced action $\sigma:=K_{0}(\gamma):\Gamma\rightarrow\OAut(K_{0}(\mathbb{M}_{d}))$. If $g=[p_{j}]\in S$ and $s_{i}\in F$, note that
\[\hat\alpha_{s_{i}}(g)=\hat\alpha_{s_{i}}([p_{j}])=[\alpha_{s_{i}}(p_{j})]=[q_{i,j}]=[\iota(q_{i,j})]=\hat\iota([q_{i,j}])\]
belongs to $\hat\iota(K_{0}(B))=H$, so we may apply $\beta$ to this element and obtain
\[\beta(\hat\alpha_{s_{i}}(g))=\hat\pi(\hat j([q_{i,j}]))=[\pi(q_{i,j})].\]
On the other hand, first applying $\beta$ followed by the action $\sigma$ gives
\begin{align*}
\sigma_{s_{i}}\circ\beta(g)&=\hat\gamma_{s_{i}}\circ\beta([p_{j}])=\hat\gamma_{s_{i}}\circ\beta([q_{1,j}])
=\hat\gamma_{s_{i}}\circ\beta(\hat\iota[q_{1,j}])\\&=\hat\gamma_{s_{i}}\circ \hat\pi\circ \hat j([q_{1,j}])
=[\gamma_{s_{i}}(\pi(q_{1,j}))].
\end{align*}

\noindent \textbf{Claim 2.} For each $i,j$, $[\pi(q_{i,j})]=[\gamma_{s_{i}}(\pi(q_{1,j}))]$ in $K_{0}(\mathbb{M}_{d})$, which will give us the desired equivariance.

Note that
\begin{align*}\|\gamma_{s_{i}}(\pi(q_{1,j}))&-\pi(\alpha_{s_{i}}(q_{1,j}))\|\leq
\|\gamma_{s_{i}}(\pi(q_{1,j}))-\gamma_{s_{i}}(\varphi(q_{1,j}))\|\\
&+\|\gamma_{s_{i}}(\varphi(q_{1,j}))-\varphi(\alpha_{s_{i}}(q_{1,j}))\|
+\|\varphi(\alpha_{s_{i}}(q_{1,j}))-\pi(\alpha_{s_{i}}(q_{1,j}))\|\\
&< \varepsilon+\delta+\varepsilon<3\varepsilon.
\end{align*}
Choosing $\varepsilon<1/6$ we guarantee that $[\gamma_{s_{i}}(\pi(q_{1,j}))]=[\pi(\alpha_{s_{i}}(q_{1,j}))]$. Now $s_{1}=e$ so we have
\begin{align*}
\|\alpha_{s_{i}}(q_{1,j})-q_{i,j}\|&\leq\|\alpha_{s_{i}}(q_{1,j})-\alpha_{s_{i}}\circ\alpha_{s_{1}}(p_{j})\|+
\|\alpha_{s_{i}}\circ\alpha_{s_{1}}(p_{j})-q_{i,j}\|\\
&\leq\|q_{1,j}-\alpha_{s_{1}}(p_{j})\|+\|\alpha_{s_{i}}(p_{j})-q_{i,j}\|<1/4+1/4=1/2.
\end{align*}
Therefore $\|\pi(\alpha_{s_{i}}(q_{1,j}))-\pi(q_{i,j})\|<1/2$ and so $[\pi(\alpha_{s_{i}}(q_{1,j}))]=[\pi(q_{i,j})]$ in $K_{0}(\mathbb{M}_{d})$ and our claim holds.\\

For a positive integer $d$, any order automorphism $\sigma$ of $(\mathbb{Z},\mathbb{Z}^{+},d)$ must be trivial. Indeed, since $\sigma$ is a positive isomorphism $\sigma(1)> 0$, and since $d=\sigma(d)=d\cdot\sigma(1)$, we must have $\sigma(1)=1$. Therefore, any action $\sigma:\Gamma \rightarrow \OAut(\mathbb{Z},\mathbb{Z}^{+},d)$ of ordered abelian groups must be trivial, and the above equivariance now reads:
\[\beta(\hat\alpha_{s_{i}}(g))=\beta(g)\qquad \forall g\in S,\quad \forall s_{i}\in F.\]

Next, we must consider an arbitrary finite subset of the positive cone $K_{0}(A)^{+}$ and not restrict ourselves to its scale $\Sigma(A)$. This is not problematic, for as $A$ is AF, its scale $\Sigma(A)$ generates the positive cone $K_{0}(A)^{+}$. So, given a finite set $S=\{[p_{1}],\dots,[p_{m}]\}\subset K_{0}(A)^{+}$, where $p_{j}\in \mathcal{P}_{\infty}(A)$, write each $[p_{j}]=\sum_{i=1}^{I_{j}}t_{ji}[q_{ji}]$ where the $t_{ji}$ are positive integers and the $q_{ji}$ are non-zero projections in $A$. Set $S'=\{[q_{ji}] ; j=1,\dots,m,\ i=1,\dots,I_{j}\}\subset\Sigma(A)$, and by our above work, obtain a suitable $H$, and $\beta$ satisfying the required conditions for the given finite set $F\subset\Gamma$. Using the fact that $\hat\alpha_{t}$ is additive for each $t\in F$, and $\beta$ is faithful on non-zero elements of $S'$ the properties $(1),(2)$, and $(3)$ of $K_{0}$-QD will hold. Since the $t_{ji}$ are positive integers, and $\beta([q_{ji}])\geq 0$, observe that for every $j$
\[0=\beta([p_{j}])=\sum_{i}t_{ji}\beta([q_{ji}])\quad \Longrightarrow\quad \beta([q_{ji}])=0,\ \forall i\quad\Longrightarrow\quad [q_{ji}]=0\ \forall i\quad\Longrightarrow\quad [p_{j}]=0,\]
so that $\beta$ is indeed faithful on non-zero elements of $S$, completing the proof.
\end{proof}

When a free group is acting on a unital AF-algebra, $K_{0}$-QD actions coincide with QD actions.

\begin{thm}\label{QDiffK0QD} Let $A$ be a unital AF-algebra and $\alpha:\mathbb{F}_{r}\rightarrow\Aut(A)$ an action, where $r\in\{1,2,\dots,\infty\}$. Then $\alpha$ is quasidiagonal if and only if $\alpha$ is $K_{0}$-QD.
\end{thm}

\begin{proof} Having shown the `only if' above, we embark on the proof of sufficiency. Denote the generators of $\mathbb{F}_{r}$ by $s_{1},\dots,s_{r}$ and set $e_{\mathbb{F}_{r}}=s_{0}$. We abbreviate $\alpha_{s_{i}}=\alpha_{i}$ for $i=0,\dots,r$, and to ease notation write $K_{0}(\alpha_{i})=\hat \alpha_{i}$ to denote the induced order automorphism at the $K_{0}$-level. Let $\delta>0$, to be determined later, and let us first consider the case where we are given a finite set of non-zero projections $p_{1}, \dots,p_{n}$ belonging to a finite dimensional subalgebra $B\subset A$. Find $\delta'=\delta'(\delta,\dim(B))$ as in Lemma~\ref{mtx units pert}. The algebras $B_{i}=\alpha_{i}(B)$ are finite dimensional and admit systems of matrix units $\mathcal{E}_{i}$ for each $i$. Since $A$ is AF, there is a finite dimensional $D\subset A$ containing $B$ with $\mathcal{E}_{i}\subset_{\delta'}D$ for every $i$. Lemma~\ref{mtx units pert} then provides us with unitaries $u_{i}$ in $A$ satisfying $\|u_{i}-1\|<\delta$ and $u_{i}B_{i}u_{i}^{*}\subset D$.

Choose $F=\{s_{1},\dots,s_{r}\}$ and $S\subset K_{0}(A)^{+}$ as $S=\{[p_{1}],\dots [p_{n}], e_{1},\dots e_{k}, f_{1},\dots,f_{l}\}$, where the $e_{j}$ generate $K_{0}(B)$ and the $f_{j}$ generate $K_{0}(D)$. More precisely, $e_{j}=[e_{11}^{(j)}]$ and $f_{j}=[f_{11}^{(j)}]$ where $\{e_{s,t}^{(j)}\}$ and $\{f_{s,t}^{(j)}\}$ are appropriate systems of matrix units for $B$ and $D$ respectively. Since $\alpha$ is $K_{0}$-QD, we obtain the subgroup $H\leq K_{0}(A)$ and the group morphism $\beta:H\rightarrow\mathbb{Z}$ satisfying all the desired properties. Suppose $\beta([1_{A}])=d>0$. By composing with an isomorphism of ordered abelian groups
\[(\mathbb{Z},\mathbb{Z}^{+},d)\cong(K_{0}(\mathbb{M}_{d}), K_{0}(\mathbb{M}_{d})^{+}, [1_{\mathbb{M}_{d}}]),\]
we may assume $\beta$ takes values in $K_{0}(\mathbb{M}_{d})$ and $\beta([1_{A}])=[1_{\mathbb{M}_{d}}]$. Denote by $\iota$ the inclusion $\iota:D\hookrightarrow A$, and note that for any generator $f_{j}$ of $K_{0}(D)$ we have $\hat\iota(f_{j})=f_{j}\in S\subset H$ whence the map
\[\beta\circ\hat\iota:K_{0}(D)\rightarrow K_{0}(\mathbb{M}_{d})\]
is a well defined group homomorphism. Since $K_{0}(D)=\mathbb{Z}^{+}f_{1}+\dots+\mathbb{Z}^{+}f_{l}$, and  $\beta$ takes positive values on $S$, $\beta\circ\hat\iota$ is certainly a positive map. Also, $\beta\circ\hat\iota([1_{A}])=\beta([1_{A}])=[1_{\mathbb{M}_{d}}]$, so there is a unital $*$-homomorphism $\varphi:D\rightarrow \mathbb{M}_{d}$ with $\hat\varphi=\beta\circ\hat\iota$. Appealing to the invariance of $\beta$, we obtain
\begin{multline*}
\hat\varphi(e_{j})=\beta\circ\hat\iota(e_{j})
=\beta\circ\hat\iota([e_{11}^{(j)}])=\beta([e_{11}^{(j)}])\\
=\beta\circ\hat\alpha_{i}([e_{11}^{(j)}])=\beta([\alpha_{i}(e_{11}^{(j)})])=\beta([u_{i}\alpha_{i}(e_{11}^{(j)})u_{i}^{*}])
=\beta\circ\hat\iota([u_{i}\alpha_{i}(e_{11}^{(j)})u_{i}^{*}])\\=\hat\varphi([u_{i}\alpha_{i}(e_{11}^{(j)})u_{i}^{*}])
=[\varphi\circ\mbox{Ad}_{u_{i}}\circ\alpha_{i}(e_{11}^{(j)})]=K_{0}(\varphi\circ\mbox{Ad}_{u_{i}}\circ\alpha_{i})(e_{j}).
\end{multline*}
Therefore the homomorphisms $\varphi|_{B}$ and $\varphi\circ\Ad_{u_{i}}\circ\alpha_{i}|_{B}$ agree at the $K_{0}$ level, as morphisms from $K_{0}(B)$ to $K_{0}(\mathbb{M}_{d})$, and by the finite-dimensionality of $B$ we know that there are unitaries $v_{i}$ in $\mathbb{M}_{d}$ with
\[\Ad_{v_{i}}\circ\varphi|_{B}=\varphi\circ\Ad_{u_{i}}\circ\alpha_{i}|_{B}.\]

By the universal property of the free group, we may define an action $\gamma:\mathbb{F}_{r}\rightarrow$Aut$(\mathbb{M}_{d})$ by $\gamma_{i}:=\gamma_{s_{i}}:=\Ad_{v_{i}}$, which gives us
\[\gamma_{i}\circ\varphi|_{B}=\varphi\circ\mbox{Ad}_{u_{i}}\circ\alpha_{i}|_{B}.\]
By Arveson's extension theorem, we may extend $\varphi$ to a unital completely positive map $\varphi:A\rightarrow\mathbb{M}_{d}$. For each $p_{j}$ and each $s_{i}$ a simple estimate using the fact that $\|1-u_{i}\|<\delta$ gives
\begin{align*}
\|\gamma_{i}\circ\varphi(p_{j})-\varphi\circ\alpha_{i}(p_{j})\|&=\|\varphi\circ\mbox{Ad}_{u_{i}}\circ\alpha_{i}(p_{j})-\varphi\circ\alpha_{i}(p_{j})\|\\
&\leq \|\mbox{Ad}_{u_{i}}\circ\alpha_{i}(p_{j})-\alpha_{i}(p_{j})\|=\|u_{i}\alpha_{i}(p_{j})u_{i}^{*}-\alpha_{i}(p_{j})\|\leq 2\delta.
\end{align*}
Now $\varphi$ is multiplicative on $D$ and hence on the $p_{j}$ and is clearly injective on $\{p_{1},\dots,p_{n}\}$. Indeed, by the condition on $\beta$ and cancellation,
\[\varphi(p_{j})=0\Rightarrow\hat\varphi([p_{j}])=0\Rightarrow \beta([p_{j}])=0\Rightarrow [p_{j}]=0\Rightarrow p_{j}=0,\]
a contradiction.

We now can proceed to the general case. To verify quasidiagonality of the action, it suffices to consider a finite set of self-adjoint elements $a_{1},\dots,a_{m}\in A$ with $\|a_{j}\|\leq 1$, the finite set of standard generators $\{s_{1},\dots, s_{r}\}$ of $\mathbb{F}_{r}$ and an arbitrary $\varepsilon>0$. By Lemma~\ref{spectral lemma}, we find a finite-dimensional subalgebra $B\subset A$ such that for each $a_{j}$
\[\bigg\|a_{j}-\sum_{l=1}^{L_{j}}t_{jl}p_{jl}\bigg\|<\eta\qquad 0\neq p_{jl}\in \mathcal{P}(B),\quad p_{jl}p_{jk}=0,\ l\neq k\]
where $\eta=\eta(\varepsilon)>0$ will be determined later. Set $b_{j}=\sum_{l=1}^{L_{j}}t_{jl}p_{jl}$. Note that fixing $j$, the projections $p_{jl}$ are orthogonal, whence
\[\max_{l}|t_{jl}|=\|b_{j}\|\leq \|b_{j}-a_{j}\|+\|a_{j}\|\leq \eta+1.\]

Apply all our above work to the set of projections $\{p_{jl}\}_{j,l}\subset B$ in order to obtain $\varphi,d,\gamma$, as above for an arbitrary $\delta>0$. We estimate
\begin{align*}
\|\gamma_{i}\circ\varphi(a_{j})&-\varphi\circ\alpha_{i}(a_{j})\|\\&\leq  \|\gamma_{i}\circ\varphi(a_{j})-\gamma_{i}\circ\varphi(b_{j})\|
+\|\gamma_{i}\circ\varphi(b_{j})-\varphi\circ\alpha_{i}(b_{j})\|+\|\varphi\circ\alpha_{i}(b_{j})-\varphi\circ\alpha_{i}(a_{j})\|\\
&\leq 2\|a_{j}-b_{j}\|+\bigg\|\sum_{l=1}^{L_{j}}t_{jl}(\gamma_{i}\circ\varphi(p_{jl})-\varphi\circ\alpha_{i}(p_{jl}))\bigg\|\\
&\leq 2\eta+\sum_{l=1}^{L_{j}}|t_{jl}|\|(\gamma_{i}\circ\varphi(p_{jl})-\varphi\circ\alpha_{i}(p_{jl}))\|\\
&\leq 2\eta+\sum_{l=1}^{L_{j}}(1+\eta)2\delta=2\eta+L_{j}(1+\eta)2\delta \leq 2\eta + L(1+\eta)2\delta
\end{align*}
where $L=\max_{j}L_{j}$. To verify approximate multiplicativity, observe
\[\|a_{i}a_{j}-b_{i}b_{j}\|\leq \|a_{i}a_{j}-a_{i}b_{j}\|+\|a_{i}b_{j}-b_{i}b_{j}\|\leq \|a_{i}\|\|a_{j}-b_{j}\|+ \|a_{j}-b_{j}\|\|b_{j}\|\leq \eta +\eta(1+\eta).\]
A similar estimate yields $\|\varphi(a_{i})\varphi(a_{j})-\varphi(b_{i})\varphi(b_{j})\|\leq \eta +\eta(1+\eta)$. Note that $\varphi$, being multiplicative on all the projections $p_{jl}$, will also be multiplicative on the $b_{j}$, therefore
\[\|\varphi(a_{i}a_{j})-\varphi(a_{i})\varphi(a_{j})\|\leq \|\varphi(a_{i}a_{j})-\varphi(b_{i}b_{j})\|+\|\varphi(a_{i})\varphi(a_{j})-\varphi(b_{i})\varphi(b_{j})\|\leq 2(\eta +\eta(1+\eta)).\]

Now since $\varphi$ is faithful on the $p_{jl}$, $\varphi$ will be isometric on the $b_{j}$. Indeed, using the fact that $\varphi(p_{jl})\varphi(p_{jk})=\varphi(p_{jl}p_{jk})=0$ for $k\neq l$, we have
\[\|\varphi(b_{j})\|=\bigg\|\sum_{l=1}^{L_{j}}t_{jl}\varphi(p_{jl})\bigg\|=\max_{l}|t_{jl}|=\|b_{j}\|.\]

Finally we estimate
\[\left|\|\varphi(a_{j})\|-\|a_{j}\|\right|\leq \left|\|\varphi(a_{j})\|-\|\varphi(b_{j})\|\right|+\left|\|b_{j}\|-\|a_{j}\|\right|\leq \|\varphi(a_{j})-\varphi(b_{j})\|+\|a_{j}-b_{j}\|\leq 2\eta.\]

We need only choose the right $\eta$ and $\delta$. Given $\varepsilon>0$, choose $\eta$ so that $\eta<\varepsilon/4$, and $2(\eta +\eta(1+\eta))<\varepsilon$. Then simply choose $\delta<\varepsilon/(4L(1+\eta))$. By our above estimates this choice will ensure the approximate equivariance $\|\gamma_{i}\circ\varphi(a_{j})-\varphi\circ\alpha_{i}(a_{j})\|<\varepsilon$, the approximate multiplicativity $\|\varphi(a_{i}a_{j})-\varphi(a_{i})\varphi(a_{j})\|<\varepsilon$, and the approximate isometricity $|\|\varphi(a_{j})\|-\|a_{j}\||<\varepsilon$, so that $\mathbb{F}_{r}\curvearrowright A$ is quasidiagonal.
\end{proof}

Combining the last few results with Theorem~\ref{KerrNowak} we obtain:

\begin{cor} Let $\alpha:\mathbb{F}_{r}\rightarrow\Aut(A)$ be an action on a unital AF algebra. The following are equivalent:
\begin{enumerate}
\item $\alpha$ is MF.
\item $\alpha$ is QD.
\item $\alpha$ satisfies $K_{0}$-QD.
\item The reduced crossed product $A\rtimes_{\lambda,\alpha}\mathbb{F}_{r}$ is an MF algebra.
\end{enumerate}
\end{cor}

\begin{proof} $(1)\Leftrightarrow(2)\Leftrightarrow(4)$ is Theorem~\ref{Summarypart1}.

$(2)\Leftrightarrow (3)$: This is Theorem~\ref{QDiffK0QD}.

\end{proof}

We seek yet another equivalent K-theoretic condition, this time in the spirit of a coboundary subgroup analogous to N. Brown's main result in~\cite{B}. Now we insist that our discrete group be a free group $\Gamma=\mathbb{F}_{r}=\langle s_{1},\dots,s_{r}\rangle$ of finitely many generators, which acts on a unital AF algebra $A$. Denote this action by $\alpha$ and $\alpha_{i}=\alpha_{s_{i}}$. By the Pimsner-Voiculescu six term exact sequence (consult~\cite{Bl} p.78) and the fact that $K_{1}(A)=\{0\}$ for an AF algebra, the sequence
\[\bigoplus_{j=1}^{r}K_{0}(A)\stackrel{\sigma}{\longrightarrow} K_{0}(A)\stackrel{\hat\iota}{\longrightarrow} K_{0}(A\rtimes_{\lambda,\alpha}\mathbb{F}_{r})\longrightarrow 0\]
is exact, where $\iota:A\hookrightarrow A\rtimes_{\lambda,\alpha}\mathbb{F}_{r}$ is the canonical inclusion, and
\[\sigma(g_{1},\dots,g_{r})=\sum_{j=1}^{r}(g_{j}-\hat\alpha_{j}(g_{j})).\]
Write $H_{\sigma}=\im(\sigma)\leq K_{0}(A)$, so that $K_{0}(A)/H_{\sigma}\cong K_{0}(A\rtimes_{\lambda,\alpha}\mathbb{F}_{r})$. First, a preliminary result about the subgroup $H_{\sigma}$.

\begin{lem}\label{Hsigma} In the above context, the subgroup $H_{\sigma}\leq K_{0}(A)$ is generated by the set \[\{g-\hat\alpha_{w}(g) : g\in K_{0}(A), w\in\mathbb{F}_{r}\}.\]
\end{lem}

\begin{proof} One direction being clear from the definition, we claim that every element of the form $g-\hat\alpha_{w}(g)$ will belong to $H_{\sigma}$. To that end, write the alphabet for $\mathbb{F}_{r}$ as
\[\mathcal{A}=\{e,s_{1},\dots s_{r},s_{1}^{-1},\dots,s_{r}^{-1}\}.\]
First note that for every letter $a\in\mathcal{A}$, and for all $g\in K_{0}(A)$, $g-\hat\alpha_{a}(g)\in H_{\sigma}$. For $a=e$ it's clear. Suppose $a=s_{i}$, for some $1\leq i\leq r$, then $g-\hat\alpha_{s_{i}}(g)=\sigma(0,\dots,0,g,0,\dots,0)\in H_{\sigma}$, where $g$ is in the $i$th spot. Next, say $a=s_{i}^{-1}$ for some $1\leq i\leq r$, then
\begin{align*}
g-\hat\alpha_{s_{i}^{-1}}(g)&=\widehat{\alpha_{s_{i}}\circ\alpha_{s_{i}^{-1}}}(g)-\hat\alpha_{s_{i}^{-1}}(g)
=\hat\alpha_{s_{i}}\circ\hat\alpha_{s_{i}^{-1}}(g)-\hat\alpha_{s_{i}^{-1}}(g)\\&=\hat\alpha_{s_{i}}(\hat\alpha_{s_{i}^{-1}}(g))-\hat\alpha_{s_{i}^{-1}}(g)
=-(\hat\alpha_{s_{i}^{-1}}(g)-\hat\alpha_{s_{i}}(\hat\alpha_{s_{i}^{-1}}(g)))=-(f-\hat\alpha_{s_{i}}(f))\in H_{\sigma},
\end{align*}
where $f=\hat\alpha_{s_{i}^{-1}}(g)$. Now let $w\in\mathbb{F}_{r}$ be a (reduced) word in symbols from $\mathcal{A}$. We have shown that if $|w|=1$ the claim holds, so proceed by strong induction on $|w|$. If $|w|=l$, write $w=aw'$ where $a\in\mathcal{A}\setminus\{e\}$ so $|w'|<l$. For $g\in K_{0}(A)$:
\begin{align*}
g-\hat\alpha_{w}(g)&=g-\hat\alpha_{aw'}(g)=g-\hat\alpha_{w'}(g)+\hat\alpha_{w'}(g)-\hat\alpha_{a}\circ\hat\alpha_{w'}(g)\\
&=g-\hat\alpha_{w'}(g)+\hat\alpha_{w'}(g)-\hat\alpha_{a}(\hat\alpha_{w'}(g))=g-\hat\alpha_{w'}(g)+f-\hat\alpha_{a}(f)\in H_{\sigma}
\end{align*}
by the inductive hypothesis, where $f=\hat\alpha_{w'}(g)$. This completes the proof.
\end{proof}

We shall make use of the following key lemma which is due to Spielberg. Consult~\cite{Sp} for a clear argument. Note that this result relies on the theorem of Effros, Handelman and Shen~\cite{EHS} on dimension groups.

\begin{lem}[Spielberg]\label{Spiel} If $K$ is a dimension group and $H$ is a subgroup of $K$ with $H\cap K^{+}=\{0\}$, then there is a dimension group $G$ and a positive group homomorphism $\theta:K\rightarrow G$ such that
\begin{enumerate}
\item $H\subset\ker(\theta)$,
\item $\ker(\theta)\cap K^{+}=\{0\}$.
\end{enumerate}
\end{lem}

\begin{prop}\label{KoEquiv}Let $\alpha:\mathbb{F}_{r}\rightarrow \Aut(A)$ be an action of a free group on a unital AF algebra. Then the following are equivalent.
\begin{enumerate}
\item $\alpha$ is $K_{0}$-QD.
\item $H_{\sigma}\cap K_{0}(A)^{+}=\{0\}$.
\end{enumerate}
\end{prop}

\begin{proof} $(1)\Rightarrow (2)$: Suppose $x=\sum_{j=1}^{r}(g_{j}-\hat\alpha_{j}(g_{j}))>0$ in $K_{0}(A)^{+}$. For each $j$ write $g_{j}=x_{j}-y_{j}$ with $x_{j},y_{j}\in K_{0}(A)^{+}$. By setting $S=\{x_{1},\dots,x_{r},y_{1},\dots, y_{r},x\}$ and $F=\{s_{1},\dots, s_{r}\}$ as our finite sets, we obtain a suitable $H$ and $\beta:H\rightarrow \mathbb{Z}$ with the desired conditions in the definition of $K_{0}$-QD. Observe then that
\[0<\beta(x)=\beta\bigg(\sum_{j=1}^{r}(g_{j}-\hat\alpha_{j}(g_{j}))\bigg)=\sum_{j=1}^{r}(\beta(g_{j})-\beta(\hat\alpha_{j}(g_{j})))=\sum_{j=1}^{r}(\beta(g_{j})-\beta(g_{j}))=0,\]
a contradiction. Therefore, $x=0$, and $(2)$ holds.\\

$(2)\Rightarrow (1)$: Since $H_{\sigma}\cap K_{0}(A)^{+}=\{0\}$ using Lemma~\ref{Spiel} we get a dimension group $(G,G^{+})$ and positive group homomorphism $\theta:K_{0}(A)\rightarrow G$ satisfying $H_{\sigma}\subset\ker(\theta)$ and $\ker(\theta)\cap K_{0}(A)^{+}=\{0\}$. Given finite subsets $F\subset\mathbb{F}_{r}$ and $S\subset K_{0}(A)^{+}$, consider the finitely generated subgroup $H$ of $K_{0}(A)$ given by
\[H=\langle\hat\alpha_{s}(x) : x\in S\cup\{[1]\}, s\in F\cup\{e\}\rangle.\]
This $H$ will be the desired subgroup for verifying that $\alpha$ is $K_{0}$-QD, and so what is needed is the correct $\beta:H\rightarrow\mathbb{Z}$. Restricting $\theta$ to $H$  we note that the subgroup $\theta(H)\leq G$ is also generated by the finitely many positive elements  $\theta(\hat\alpha_{t}(x))$ for $x\in S\cup\{[1]\}$ and $t\in F\cup\{e\}\}$. To ease notation, label these as $k_{1},\dots ,k_{n}\in G^{+}$. Since $G$ is a dimension group, write $(G,(\beta_{i}))$ for the limit of an inductive sequence of ordered abelian groups $(G_{i},G_{i}^{+})$
\[G_{1}\stackrel{h_{1}}{\longrightarrow}G_{2}\stackrel{h_{2}}{\longrightarrow}G_{3}\stackrel{h_{3}}{\longrightarrow}\dots\]
where the $h_{i}$ are positive group homomorphisms, the $\beta_{i}:G_{i}\rightarrow G$ are the connecting positive group homomorphisms and each $(G_{i},G_{i}^{+})$ is order isomorphic to $(\mathbb{Z}^{p_{i}},\mathbb{Z}^{p_{i}}_{\geq0})$ for some positive integers $p_{i}$. There is an $m$ large enough so that $k_{1},\dots,k_{n}\in\beta_{m}(G_{m}^{+})$. Set $k_{i}=\beta_{m}(y_{i})$ for some $y_{i}\in G_{m}^{+}$. The group $G_{m}$ is abelian and finitely generated, and so is its subgroup $K=\ker(\beta_{m})$, say $K=\langle g_{1},\dots ,g_{l}\rangle$. Now choose $k$ large enough so that $k\geq m$ and such that $h_{k,m}(g_{j})=0$ for all $j=1,\dots, l$. Identify $(G_{k},G_{k}^{+})=(\mathbb{Z}^{p},\mathbb{Z}^{p}_{\geq0})$ for some $p\in\mathbb{N}$ and define $\psi:G_{k}\rightarrow\mathbb{Z}$ by
\[\psi((z_{1},\dots,z_{p}))=\sum_{i=1}^{p}z_{i}.\]
Clearly, $\psi$ is a positive group homomorphism which is faithful on the positive cone $G_{k}^{+}$. We now may define $\phi:\beta_{m}(G_{m})\rightarrow\mathbb{Z}$ by $\phi(\beta_{m}(g)):=\psi(h_{k,m}(g))$. Observe that, by our choice of $k$,
\begin{align*}\beta_{m}(g)= \beta_{m}(g')&\Leftrightarrow g-g'\in K \Rightarrow h_{k,m}(g-g')=0 \Leftrightarrow h_{k,m}(g)=h_{k,m}(g')\\&\Leftrightarrow \psi(h_{k,m}(g))=\psi(h_{k,m}(g'))\Leftrightarrow\phi(\beta_{m}(g))=\phi(\beta_{m}(g')),
\end{align*}
verifying that $\phi$ is well defined. It is routine to check that $\phi$ is additive on $\beta_{m}(G_{m})$. Naturally, we now compose and define
\[\beta:=\phi\circ\theta|_{H}:H\longrightarrow\mathbb{Z}.\]
Since the $t_{i}$ lie in $\beta_{m}(G_{m})$, we have $\theta(H)=\langle t_{1},\dots, t_{n}\rangle\leq\beta_{m}(G_{m})$ and thus $\beta$ is a well defined group homomorphism. Let $x\in S$. Then from our notation $\theta(x)=t_{i}$ for some $i$. So
\[\beta(x)=\phi(\theta(x))=\phi(t_{i})=\phi(\beta_{m}(y_{i}))=\psi(h_{k,m}(y_{i}))\geq0,\]
since $\psi$ and $h_{k,m}$ are positive maps and $y_{i}\in G_{m}^{+}$. To see that $\beta$ is faithful on $S$, we use the fact that $\psi$ is faithful on $G_{k}^{+}$: if $x\in S$ then
\begin{align*}
\beta(x)=0&\Rightarrow \psi(h_{k,m}(y_{i}))=0\Rightarrow h_{k,m}(y_{i})=0\Rightarrow \beta_{k}(h_{k,m}(y_{i}))=0\Rightarrow \beta_{m}(y_{i})=0\\ &\Rightarrow t_{i}=0\Rightarrow \theta(x)=0\Rightarrow x\in\ker(\theta)\cap K_{0}(A)^{+}\Rightarrow x=0.
\end{align*}

Finally, we verify the invariance of $\beta$. For $x\in S$ and $s\in F$, Lemma~\ref{Hsigma} ensures that $x-\hat\alpha_{s}(x)$ belongs to $H_{\sigma}$, which  in turn lives inside $\ker(\theta)$, so that $\theta(x-\hat\alpha_{s}(x))=0$. Therefore,
\[\beta(x-\hat\alpha_{s}(x))=\phi\circ\theta(x-\hat\alpha_{s}(x))=0 \Longrightarrow \beta(x)=\beta(\hat\alpha_{s}(x)),\]
completing the proof.

\end{proof}

While MF algebras are always stably finite, the authors of~\cite{BK} remarked that there are no known examples of stably finite \cstar-algebras which are not MF. With the right $K_{0}$ condition at our disposal we can give an answer to this inquiry for a special class of crossed product algebras. Here is the crucial result, reminiscent of N. Brown's main result in~\cite{B}.

\begin{thm}\label{Main2} Let $A$ be a unital AF algebra and $\alpha:\mathbb{F}_{r}\rightarrow\Aut(A)$ an action of the free group on $r$ generators. Then the following are equivalent.

\begin{enumerate}
\item $\alpha$ is MF.
\item $\alpha$ is quasidiagonal.
\item $\alpha$ is $K_{0}$-QD.
\item $H_{\sigma}\cap K_{0}(A)^{+}=\{0\}.$
\item The reduced crossed product $A\rtimes_{\lambda,\alpha}\mathbb{F}_{r}$ is MF.
\item The reduced crossed product $A\rtimes_{\lambda,\alpha}\mathbb{F}_{r}$ is stably finite.
\end{enumerate}

\end{thm}

\begin{proof} For such an action, the equivalences $(1)\Leftrightarrow(2)\Leftrightarrow(3)\Leftrightarrow(4)\Leftrightarrow(5)$ are contained in Theorem~\ref{Summarypart1}, Propositions~\ref{Main1} and~\ref{KoEquiv}. Now every MF algebra is stably finite, so it suffices to show $(6)\Rightarrow(4)$. To that end, suppose $x\in H_{\sigma}\cap K_{0}(A)^{+}$. Then $x=[p]$ where $p\in\mathcal{P}_{\infty}(A)$. From the Pimsner-Voiculescu exact sequence, $\ker(\hat\iota)=H_{\sigma}$, so $0=\hat\iota(x)=\hat\iota([p]=[\iota(p)]$ in $K_{0}(A\rtimes_{\lambda,\alpha}\mathbb{F}_{r})$. However, the stable finiteness of $A\rtimes_{\lambda,\alpha}\mathbb{F}_{r}$  ensures $\iota(p)=0$, which implies that $p=0$ since $\iota$ is inclusion. Thus $x=0$ and $(3)$ holds.

\end{proof}

\begin{example} If $A$ is an AF-algebra and $(A,\mathbb{F}_{r},\alpha)$, $(A,\mathbb{F}_{r},\beta)$ are \cstar-dynamical systems which agree on $K$-theory, that is $\hat\alpha=\hat\beta$, then Theorem~\ref{Main2} ensures that $\alpha$ is MF if and only if $\beta$ is MF. In particular, recall that actions $\alpha$ and $\beta$ are said to be \emph{exterior equivalent} provided there is a map $u:\Gamma\rightarrow\mathcal{U}(A)$ which satisfies the cocycle condition $u_{st}=u_{s}\alpha_{s}(u_{t})$ and $\beta_{s}=\Ad_{s}\circ\alpha_{s}$ for each $s,t\in\Gamma$. In this case $\alpha$ and $\beta$ clearly agree on $K$-theory and the above discussion applies.
\end{example}

\end{document}